\documentclass[12pt,a4paper]{article}
\usepackage{amsmath}
\usepackage{amssymb}
\usepackage{amsthm}
\usepackage{amsfonts}
\usepackage{latexsym}

\theoremstyle{plain}
\newtheorem{thm}{Theorem}[section]
\newtheorem{cor}{Corollary}[section]
\newtheorem{lem}{Lemma}[section]

\theoremstyle{definition}
\newtheorem{defn}{Definition}[section]

\theoremstyle{remark}

\newtheorem{rem}{Remark}[section]

\title{Equivariant local scaling asymptotics\\ for smoothed T\"{o}plitz
spectral projectors}
\author{Roberto Paoletti\footnote{\noindent{\bf Address:}
Dipartimento di Matematica e Applicazioni, Universit\`a degli Studi
di Milano Bicocca, Via R. Cozzi 53, 20125 Milano,
Italy; {\bf e-mail}: roberto.paoletti@unimib.it }}
\date{}

\begin{document}

\maketitle
\begin{abstract}
Let $X$ be the unit circle bundle of a positive line bundle on a Hodge
manifold.
We study the local scaling asymptotics of the smoothed spectral projectors associated to a
first order
elliptic T\"{o}plitz operator $T$ on $X$, 
possibly in the presence of Hamiltonian symmetries. The resulting expansion is then used to give a local 
derivation of an equivariant Weyl law. It is not required
that $T$ be invariant under the structure circle action, that is, $T$ needn't be a Berezin-T\"{o}plitz operator.

\end{abstract}
 
\section{Introduction}

Let $(M,J,\omega)$ be a compact complex $d$-dimensional Hodge manifold, and let 
$(A,h)$ be a positive holomorphic line bundle on $M$, such that the unique compatible
covariant derivative on $A$ has curvature $\Theta=-2i\,\omega$. Let $X\subseteq A^\vee$
be the unit circle bundle in the dual line bundle, with the induced connection 1-form 
$\alpha$. Thus $(X,\alpha)$ is a contact manifold and $\mathrm{d}\alpha=2\,\pi^*(\omega)$,
where $\pi:X\rightarrow M$
is the bundle projection.
We shall consider on $M$ and $X$ the volume forms
$$
\mathrm{d}V_M=:\frac{1}{d!}\,\omega^{\wedge d}\,\,\,\,\,
\mathrm{d}V_X=:\frac{1}{2\pi}\,\alpha\wedge \pi^*(\mathrm{d}V_M),
$$
and the associated densities $|\mathrm{d}V_M|$, $|\mathrm{d}V_X|$.

With these choices, one may consider the Hilbert spaces $L^2\left(M,A^{\otimes k}\right)$
of square summable sections of powers of $A$, and the Hilbert space $L^2(X)$ of square summable
complex functions on $X$. The structure circle action on $X$ determines an equivariant splitting into
isotypes 
$$
L^2(X)\cong \bigoplus _k L^2(X)_k.
$$
As is well-known, there is for every $k\in \mathbb{N}$ a natural unitary isomorphism 
$$
L^2(X)_k\cong L^2\left(M,A^{\otimes k}\right),
$$
under which the Hilbert direct sum of the spaces of holomorphic sections 
$$
H^0\left(M,A^{\otimes k}\right)\subseteq L^2\left(M,A^{\otimes k}\right),
\,\,\,\,\,\,\,(k=0,1,2,\ldots),
$$
corresponds to the Hardy space $H(X)\subseteq L^2(X)$ (see \cite{bdm-g}, \cite{bsz} and \cite{sz} for a detailed discussion).
The orthogonal projector $\Pi:L^2(X)\rightarrow H(X)$ is known in the literature as the Szeg\"{o} projector, and its
distributional kernel as the Szeg\"{o} kernel, of $X$.

A T\"{o}plitz operator of degree $k$ (in the sense of \cite{bdm-g}) is a composition
$
T=\Pi\circ Q\circ \Pi
$, 
where $Q$ is a pseudo-differential operator of degree $k$ on $X$, regarded as an endomorphism of $H(X)$. 
By the theory in \cite{bdm-g}, we may assume without loss that $[\Pi,Q]=0$, so that
$T$ is the restriction of $Q$ to $H(X)$.

A T\"{o}plitz operator $T$ has a well-defined principal symbol $\sigma_T$, given by the restriction of the principal
symbol of $Q$ to the closed symplectic cone sprayed by $\alpha$:
$$
\Sigma=:\big\{(x,r\,\alpha_x)\,:\,x\in X,\,r>0\big\}\subseteq TX\setminus \{0\};
$$
thus $\sigma_T:\Sigma\rightarrow \mathbb{C}$ is independent of the particular choice of $Q$ \cite{bdm-g}.
One calls $T$ \textit{elliptic} if $\sigma_T$ is the restriction of an elliptic symbol; if $T$ is elliptic,
one may assume in addition to the above that $Q$ is also elliptic.

We shall say that $T$ is self-adjoint to mean that it is formally self-adjoint with respect to the $L^2$-product
on $X$ associated to $\mathrm{d}V_X$. In this case, $Q$ itself may be assumed to be self-adjoint, and
$\sigma_T$ is real-valued. 

For instance, given $h\in \mathcal{C}^\infty(X)$, one obtains a zeroth-order self-adjoint 
T\"{o}plitz operator $T(h)$ by taking $Q=M_h$, the multiplication operator by $h$.
Then $\varsigma_{T(h)}=h$, so that $T(h)$ is elliptic precisely when $h$ is nowhere vanishing.
Clearly, $T(h)$ is self-adjoint if and only if 
$h$ is real-valued.

\begin{defn}
Let $T$ be a T\"{o}plitz operator as above. The \textit{reduced symbol} of $T$ is the function
$\varsigma_T:X\rightarrow \mathbb{C}$ defined by
$$
\varsigma_T(x)=:\sigma_T(x,\alpha_x)\,\,\,\,\,\,\,\,\,(x\in X).
$$
\end{defn}

So let $T$ be a first order self-adjoint first order T\"{o}plitz operator on $X$ with positive reduced symbol
$\varsigma_T>0$. Let $\lambda_1\le \lambda_2\le\cdots$ be the eigenvalues of $T$, repeated according to multiplicity, so that
$\lambda_j\uparrow +\infty$. Let $(e_j)$ be any complete orthonormal system of $H(X)$, with $e_j$ eigensection 
associated to $\lambda_j$. 
Thus, for any eigenvalue $\eta\in \mathbb{R}$ the $\mathcal{C}^\infty$ function
on $X\times X$ given by
$$
P_\eta(x,y)=:\sum _{j:\lambda_j=\eta}e_j(x)\cdot \overline{e_j(y)}\,\,\,\,\,\,\,\,\,\,\,(x,y\in X)
$$
is the Schwartz kernel of the $L^2$-orthogonal projector onto the eigenspace $H(X,T)_\eta\subseteq H(X)$ 
of $T$ associated to $\eta$.
If $I\subseteq \mathbb{R}$ is any bounded interval, we may view it as a spectral band and
similarly consider the corresponding spectral projector
$
P_I=:\sum _{\eta\in I}P_\eta$;
this is a smoothing operator, with kernel 
$$
P_I(x,y)=:\sum _{j:\lambda_j\in I}e_j(x)\cdot \overline{e_j(y)}\,\,\,\,\,\,\,\,\,\,\,(x,y\in X).
$$
If $H(X,T)_I$ is the range of $P_I$, its dimension is the number of $\lambda_j\in I$. 
Thus if $\lambda\in \mathbb{R}$ the trace 
$P_{\lambda+I}$ is the number of eigenvalues of $T$ within the spectral band $I_\lambda=\lambda+I$,
drifting to infinity as $\lambda\rightarrow+\infty$; locally on $X\times X$, 
$P_{\lambda+I}(x,y)$ encapsulates the asymptotic
concentration behavior of the eigensections of $T$ pertaining to the band $I_\lambda$
traveling to infinity.

In practice, rather than dealing directly with the $P_I$'s, 
after \cite{hor} and \cite{dg} one considers the
approximations obtained by replacing the characteristic function of $I$ by 
a $\mathcal{C}^\infty$ function $\gamma:\mathbb{R}\rightarrow\mathbb{R}$ of rapid decrease.
Thus one defines
\begin{equation*}
 \mathcal{P}_\gamma(x,y)=:\sum_j\gamma(\lambda_j)\,e_j(x)\cdot \overline{e_j(y)}\,\,\,\,\,\,\,\,\,\,\,(x,y\in X).
\end{equation*}
Then $\mathcal{P}_\gamma$ is again a smoothing operator, given by a smoothed average of the
$P_\eta$'s (see the discussion in \cite{gr_sj}).
The analogue of $P_{\lambda+I}$ is then given by
$\mathcal{P}_{\gamma_\lambda}$, where $\gamma_\lambda=\gamma(\cdot -\lambda)$.

A convenient description of these smoothly averaged spectral projectors is as \lq smoothed T\"{o}plitz
wave operators\rq, as follows. For $\tau\in \mathbb{R}$, 
let $U_T(\tau)=:e^{i\tau T}=\Pi\circ e^{i\tau Q}\circ \Pi$; thus $U_T(\tau)$ is
the unitary endomorphism of $H(X)$ given by the restriction of $e^{i\tau Q}$, and has
distributional kernel 
\begin{equation}
 \label{eqn:toplitz_wave_operator}
U_T(\tau)(x,y)=:\sum_je^{i\tau\lambda_j}\,e_j(x)\cdot \overline{e_j(y)}\,\,\,\,\,\,\,\,\,\,\,(x,y\in X).
\end{equation}

For any
$\chi\in \mathcal{S}(\mathbb{R})$ (function of rapid decrease) the averaged operator
\begin{equation}
 \label{eqn:average_wave_operator}
 S_{\chi}=:\int_{-\infty}^{+\infty}\chi(\tau)\,U_T(\tau)\,\mathrm{d}\tau
\end{equation}
is a smoothing operator, with Schwartz kernel
\begin{equation}
\label{eqn:kernel_average_wave_operator}
S_{\chi}(x,y)=:\sum_j\widehat{\chi}(-\lambda_j)\,\,e_j(x)\cdot \overline{e_j(y)}\,\,\,\,\,\,\,\,\,\,\,(x,y\in X).
\end{equation}
Thus in the previous notation
$S_{\chi}=\mathcal{P}_\gamma$, with $\gamma=\widehat{\chi}(-\cdot)$ (here $\widehat{\chi}$ is the Fourier transform of $\chi$).
If $\chi$ is replaced by $\chi\cdot e^{-i\lambda(\cdot)}$, we get
\begin{equation}
 \label{eqn:drifting_projector}
 S_{\chi\cdot e^{-i\lambda(\cdot)}}(x,y)
 =\sum_j\widehat{\chi}(\lambda-\lambda_j)\,\,e_j(x)\cdot \overline{e_j(y)}\,\,\,\,\,\,\,\,\,\,\,(x,y\in X).
\end{equation}
That is, $S_{\chi\cdot e^{-i\lambda(\cdot)}}=\mathcal{P}_{\gamma_\lambda}$. In particular, 
$$
\mathrm{trace}\left(S_{\chi\cdot e^{-i\lambda(\cdot)}}\right)=\sum_j\widehat{\chi}(\lambda-\lambda_j);
$$
an asymptotic estimate on the latter trace leads, by a Tauberian argument, to a Weyl law for $T$.
In \cite{p_weyl} a pointwise asymptotic estimate on the diagonal restriction $S_{\chi\cdot e^{-i\lambda(\cdot)}}(x,x)$
was given for $\lambda\rightarrow+\infty$, leading  (in this special setting)
to a local proof of the Weyl law for T\"{o}plitz operators
in \cite{bdm-g}. In the present paper, we shall look at the near diagonal scaling asymptotics of 
$S_{\chi\cdot e^{-i\lambda(\cdot)}}$. We shall also consider similar asymptotics
for the equivariant versions of these operators, arising 
in the presence of quantizable Hamiltonian actions on $(M,J,\omega)$.

Suppose that $G$ is a connected compact Lie group of real dimension $e$, and
 that $\mu^M:G\times M\rightarrow M$ is a holomorphic
and Hamiltonian action, with moment map $\Phi:M\rightarrow \mathfrak{g}^\vee$, where $\mathfrak{g}$ is the Lie algebra of $G$.
Also, assume that $\mu$ can be linearized to a metric preserving holomorphic action of $G$ on $(A,h)$, so that by restriction we obtain
an action of $G$ on $X$, $\mu^X:G\times X\rightarrow X$. 
This can always be done infinitesimally: if $\xi\in \mathfrak{g}$, let $\xi_M\in \mathfrak{X}(M)$ be the vector field on $M$
induced by $\xi$ under $\mu$, and let $\Phi_\xi=:\langle \Phi,\xi\rangle$ be the $\xi$-component of $\Phi$; then 
\begin{equation}
 \label{eqn:contact_vector_field}
\xi_X=:\xi_M^\sharp-\Phi_\xi\,\frac{\partial}{\partial \theta}
\end{equation}
is a contact vector field on $X$, lifting $\xi_M$ under $\mathrm{d}\pi$. Here $\upsilon^\sharp\in\mathfrak{X}(X)$ is the horizontal
lift of $\upsilon\in \mathfrak{X}(M)$, and $\partial/\partial\theta\in \mathfrak{X}(X)$ is the infinitesimal generator of the structure
$S^1$-action on $X$. Thus the obstruction to the existence of a global lifting is of topological nature.

In view of the compatibility assumptions on $\mu^M$, $G$ acts on $X$ under $\mu^X$ as a group of contactomorphims and  
leaves the Hardy space invariant; hence there is a naturally induced unitary representation $\widetilde{\mu}:G\rightarrow U\big(H(X)\big)$,
where $\widetilde{\mu}_g(f)=:f\circ \mu^X_{g^{-1}}$. Let the set $\{\varpi\}$ label 
the collection of all irreducible characters $\chi_\varpi$
of $G$, associated to the (finite dimensional) unitary representations $\left(\rho_\varpi,V^{(\varpi)}\right)$. 

By the Peter-Weyl theorem, if $\rho:G\rightarrow U(H)$ is 
a unitary representation of $G$ on any Hilbert space $H$, then there is a natural equivariant and orthogonal 
Hilbert space direct sum decomposition
$$
H=\bigoplus_{\varpi}H^{(\varpi)},
$$
where $H^{(\varpi)}\subseteq H$ is a closed subspace, unitarily and equivariantly isomorphic to a Hilbert space
direct sum of copies of $V^{(\varpi)}$.
Correspondingly, in our case we obtain a Hilbert space direct sum decomposition
$$
H(X)=\bigoplus_{\varpi}H(X)^{(\varpi)}.
$$

Now suppose that $T$ is a $G$-invariant self-adjoint T\"{o}plitz operator on $X$,
with positive symbol $\varsigma_T>0$. Then for every eigenvalue $\eta\in \mathrm{Spec}(T)$
the eigenspace $H(X,T)_\eta\subseteq H(X)$ is a finite-dimensional unitary $G$-representation, 
and so it also admits an equivariant direct sum decomposition into isotypical components:
$$
H(X,T)_\eta =\bigoplus_{\varpi}H(X,T)_\eta^{(\varpi)};
$$
here $H(X,T)_\eta^{(\varpi)}=H(X,T)_\eta\cap H(X)^{(\varpi)}$ is, for any fixed $\eta$, the null space for almost every $\varpi$.
Changing point of view, every $H(X)^{(\varpi)}$ is invariant under $T$, and so it splits equivariantly as a Hilber space direct sum
\begin{equation}
 \label{eqn:varpi_eta}
H(X)^{(\varpi)}=\bigoplus_{\eta}H(X,T)_\eta^{(\varpi)}.
\end{equation}

Let $T^{(\varpi)}:H(X)^{(\varpi)}\rightarrow H(X)^{(\varpi)}$ be restriction of $T$, and let
$\lambda_{1}^{(\varpi)}\le \lambda_{2}^{(\varpi)}\le \cdots$ be the eigenvalues of $T^{(\varpi)}$,
repeated according to multiplicity.
Let $(e_{j}^{(\varpi)})$ be any complete orthonormal system of $H(X)^{(\varpi)}$ 
such that $T^{(\varpi)}(e_{j}^{(\varpi)})=\lambda_{j}^{(\varpi)}\,e_{j}^{(\varpi)}$.
Then the equivariant analogue of (\ref{eqn:toplitz_wave_operator}) is its restriction to $H(X)^{(\varpi)}$:
\begin{equation}
 \label{eqn:equivariant_toplitz_wave_operator}
U_T^{(\varpi)}(\tau)(x,y)=:\sum_je^{i\tau\lambda_j^{(\varpi)}}\,e_j^{(\varpi)}(x)\cdot \overline{e_j^{(\varpi)}(y)}\,\,\,\,\,\,\,\,\,\,\,(x,y\in X).
\end{equation}
Similarly, we may consider its traveling smoothly averaged version,
\begin{equation}
 \label{eqn:traveling_average_wave_operator}
 S_{\chi\cdot e^{-i\lambda(\cdot)}}^{(\varpi)}=:\int_{-\infty}^{+\infty}\chi(\tau)\,e^{-i\lambda\tau}\,
U_T^{(\varpi)}(\tau)\,\mathrm{d}\tau,
\end{equation}
which is again a smoothing operator, with Schwartz kernel
\begin{equation}
 \label{eqn:equivariant_average_wave_operator}
S_{\chi\cdot e^{-i\lambda(\cdot)}}^{(\varpi)}(x,y)=:
\sum_j\widehat{\chi}\big(\lambda-\lambda_j^{(\varpi)}\big)\,\,e_j^{(\varpi)}(x)\cdot \overline{e_j^{(\varpi)}(y)}
\end{equation}
In particular, 
$$
\mathrm{trace}\left(S_{\chi\cdot e^{-i\lambda(\cdot)}}^{(\varpi)}\right)=\sum_j\widehat{\chi}\left(\lambda
-\lambda_j^{(\varpi)}\right),
$$

Thus the local asymptotics of $S_{\chi\cdot e^{-i\lambda(\cdot)}}^{(\varpi)}(\cdot,\cdot)$ captures the collective concentration behavior of
the eigensections $e_j^{(\varpi)}$'s, while its trace controls the asymptotic distribution of the eigenvalues $\lambda_j^{(\varpi)}$'s.

In the following, let us adopt the following notation:
$$
M'=:\Phi^{-1}(\mathbf{0})\subseteq M,\,\,\,\,\,\,\,\,\,\,X'=:\pi^{-1}(M')\subseteq X.
$$

Furthermore, it is convenient to make a more specific choice for a cut-off function.

\begin{defn}
Fix $\epsilon>0$. 
A \textit{good $\epsilon$-cut-off} is a function $\chi\in \mathcal{C}^\infty_0\big((-\epsilon,\epsilon)\big)$such that
\begin{enumerate}
\item  $\chi\ge 0$, $\chi(0)=1$;
\item $\widehat{\chi}\ge 0$.
\end{enumerate}
\end{defn}

It is well-known that such cut-offs exist (\S \ref{sct:cut_offs}). The condition $\chi(0)=1$ is simply a convenient
normalization.
Our first result is the following:
\begin{thm}
 \label{thm:rapid_decay}
Let $T$ be a $G$-invariant self-adjoint first oder T\"{o}plitz operator on $X$ with $\varsigma_T>0$, and let $\varpi$ be an
irreducible character of $G$. Suppose $\epsilon>0$ is small enough and $\chi$ is a good $\epsilon$-cut-off.
Then the following holds.
\begin{enumerate}
 \item Uniformly in $x,y\in X$, for $\lambda\rightarrow -\infty$ we have 
$$
S_{\chi\cdot e^{-i\lambda(\cdot)}}^{(\varpi)}(x,y)=O\left(\lambda^{-\infty}\right).
$$
\item For any $C,\epsilon>0$, uniformly for 
$$
\min\Big\{\mathrm{dist}_X(x,X'),\mathrm{dist}_X(y,X')\Big\}\ge C\,\lambda^{-\frac{11}{24}}
$$
we have for $\lambda\rightarrow+\infty$:
$$
S_{\chi\cdot e^{-i\lambda(\cdot)}}^{(\varpi)}(x,y)=O\left(\lambda^{-\infty}\right).
$$
\end{enumerate}
\end{thm}

Here $\lambda^{-\frac{11}{24}}$ might be replaced by $\lambda^{a-\frac{1}{2}}$ for any $a>0$. We set $a=1/24$ to fix ideas and because it
is a convenient choice in the following.

We shall next consider the asymptotics of $S_{\chi\cdot e^{-i\lambda(\cdot)}}^{(\varpi)}(x,y)$ 
for $x,y\rightarrow X'$, and $\mathrm{dist}_X(x,y)\rightarrow 0$. By the same approach, one might more generally consider the case 
$\mathrm{dist}_X\big(x,\mu_g(y)\big)\rightarrow 0$
for some $g\in G$, but for the sake of simplicity we shall restrict ourselves to near-diagonal asymptotics. 
More specifically,
we shall consider the asymptotics of $S_{\chi\cdot e^{-i\lambda(\cdot)}}^{(\varpi)}(x',x'')$
for $x',x''\rightarrow x\in X'$ at a controlled rate. We shall think of $x'$ and $x''$ as obtained
from $x$ by small displacements along tangent directions, and the scaling asymptotics will be controlled by the
geometry of these directions with respect to $X'$ and the $G$-orbit. It is then in order to label the various components
of the displacements that go into the statement.

The connection $\alpha$ determines a direct sum decomposition of the tangent bundle of $X$,
$
TX=\mathcal{V}\oplus \mathcal{H}
$,
where $\mathcal{V}=:\ker(\mathrm{d}\pi)=\mathrm{span}\big(\partial/\partial\theta\big)$ is the vertical tangent bundle,
and $\mathcal{H}=:\ker(\alpha)$ is the horizontal distribution. If $\upsilon\in T_xX$, we shall write accordingly
$\upsilon=\left(\upsilon',\overrightarrow{\upsilon}\right)$, where
$\upsilon'\in \mathcal{V}_x$ and $\overrightarrow{\upsilon}\in \mathcal{H}_x\cong T_mM$, where $m=\pi(x)$.

For $m\in M'$, let $N_m=:{T_mM'}^\perp\subseteq T_mM$ be the normal space to $M'$ in $M$ at $m$,
and consider the orthogonal direct sum decomposition in the tangent bundle of $M$ along $M'$:
\begin{equation}
 \label{eqn:tangent_component_decomposition}
 T_mM=:T_mM_{\mathrm{hor}}\oplus T_mM_{\mathrm{ver}}\oplus  T_mM_{\mathrm{trasv}},
\end{equation}
where:
$$
T_mM_{\mathrm{hor}}=:T_m(G\cdot m)^\perp\cap T_mM',\,\,\,
T_mM_{\mathrm{ver}}=:T_m(G\cdot m),\,\,\,T_mM_{\mathrm{trasv}}=:N_m.
$$
Here $G\cdot m\subseteq M'$ is the $G$-orbit through $m$, and $T_m(G\cdot m)$ its tangent space.
Furthermore, $T_mM_{\mathrm{hor}}$ is a complex vector subspace of $T_mM$, while 
$T_mM_{\mathrm{ver}}$ and $T_mM_{\mathrm{trasv}}$ are totally real subspaces, related by
$T_mM_{\mathrm{ver}}=J_m\big(T_mM_{\mathrm{trasv}}\big)$,
and $J_m$ is the complex structure at $m\in M$.

When $m\in M'$ we shall correspondingly write $\overrightarrow{\upsilon}\in T_mM$ as
\begin{equation}
\label{eqn:hvt}
 \overrightarrow{\upsilon}=\overrightarrow{\upsilon}_{\mathrm{h}}+\overrightarrow{\upsilon}_{\mathrm{v}}+
\overrightarrow{\upsilon}_{\mathrm{t}},
\end{equation}
where $\overrightarrow{\upsilon}_{\mathrm{h}}\in T_mM_{\mathrm{hor}}$,
$\overrightarrow{\upsilon}_{\mathrm{v}}\in T_mM_{\mathrm{ver}}$, 
$\overrightarrow{\upsilon}_{\mathrm{t}}\in T_mM_{\mathrm{trasv}}$.

The components of the displacements from $x$ to $x'$ and from $x$ to $x''$
will control the scaling asymptotics by certain \lq universal exponents\rq, depending on the symplectic and Euclidean
structures at $m$, and given by quadratic functions $Q_{\mathrm{hor}}^T,\,Q_{\mathrm{tv}}^T$
in a pair 
$\big(\overrightarrow{\upsilon},\overrightarrow{\upsilon}'\big)\in T_mM\times T_mM$, that we now define.

\begin{defn}
 \label{defn:psi_2}
Let $(V,J_V)$ is a complex vector space, and suppose that
$h_V=g_V-i\omega_V$ is a positive definite Hermitian product on it; thus $g_V$ is a $J_V$-invariant
Euclidean product and $\omega_V=g\big(J_V(\cdot),\cdot\big)$ is a 
a symplectic structure. Let us define after \cite{sz} 
$
\psi_2^V:V\times V\rightarrow \mathbb{C}
$
by
$$
\psi_2^V(\mathbf{v},\mathbf{v}')=-i\,\omega_V(\mathbf{v},\mathbf{v}')-\frac 12\,\|\mathbf{v}-\mathbf{v}'\|^2
_V,
$$
where $\|\cdot \|_V$ is the norm for $g_V$. If $(V,J_V)$ is $\mathbb{C}^d$ with its standard Hermitian product,
we shall write $\psi_2^{\mathbb{C}^d}=\psi_2$.
\end{defn}

We obtain $\psi_2^M:TM\oplus TM\rightarrow \mathbb{C}$ given by
$$
\psi_2^M\left(\overrightarrow{\upsilon}_1,\overrightarrow{\upsilon}_2\right)=:\psi_2^{T_mM}\left(\overrightarrow{\upsilon}_1,\overrightarrow{\upsilon}_2\right)
$$
if $m\in M$ and $\overrightarrow{\upsilon}_j\in T_mM$.

\begin{defn}
Let us define maps (also depending on the T\"{o}plitz operator $T$)
$$Q_{\mathrm{h}}^T,\,Q_{\mathrm{vt}},\,Q_{\mathrm{vt}}^T:\left. TM\oplus TM\right|_{M'}\rightarrow \mathbb{C}$$ as follows.
Suppose $m\in M'$ and
$\overrightarrow{\upsilon}_j\in T_mM$. Then:
\begin{eqnarray}
 \label{eqn:HORIZONTAL_quadratic_term}
Q_{\mathrm{h}}^T\left(\overrightarrow{\upsilon}_1,\overrightarrow{\upsilon}_2\right)&=:&\frac{1}{\varsigma_T(x)}\,
\psi_2^M\left({\overrightarrow{\upsilon}_1}_{\mathrm{h}},{\overrightarrow{\upsilon}_2}_{\mathrm{h}}\right)\\
&=&
\psi_2^M\left(\frac{1}{\sqrt{\varsigma_T(x)}}\,{\overrightarrow{\upsilon}_1}_{\mathrm{h}},
\frac{1}{\sqrt{\varsigma_T(x)}}\,{\overrightarrow{\upsilon}_2}_{\mathrm{h}}\right).
\nonumber
\end{eqnarray}
%
\begin{eqnarray*}
 Q_{\mathrm{vt}}\left(\overrightarrow{\upsilon}_1,\overrightarrow{\upsilon}_2\right)&=&i\,\Big[
\omega_m\big({\overrightarrow{\upsilon}_1}_{\mathrm{v}},{\overrightarrow{\upsilon}_1}_{\mathrm{t}}\big)
-\omega_m\big({\overrightarrow{\upsilon}_2}_{\mathrm{v}},{\overrightarrow{\upsilon}_2}_{\mathrm{t}}\big)\Big]\\
&&-\left(\left\|{\overrightarrow{\upsilon}_1}_{\mathrm{t}}\right\|_m^2+\left\|{\overrightarrow{\upsilon}_2}_{\mathrm{t}}\right\|^2_m\right),
\end{eqnarray*}
\begin{eqnarray}
 \label{eqn:Q_vt}
Q_{\mathrm{vt}}^T\left(\overrightarrow{\upsilon}_1,\overrightarrow{\upsilon}_2\right)&=:&\frac{1}{\varsigma_T(x)}\,
Q_{\mathrm{vt}}\left(\overrightarrow{\upsilon}_1,\overrightarrow{\upsilon}_2\right)
\nonumber\\
&=&
Q_{\mathrm{vt}}\left(\frac{1}{\sqrt{\varsigma_T(x)}}\,\overrightarrow{\upsilon}_1,
\frac{1}{\sqrt{\varsigma_T(x)}}\,\overrightarrow{\upsilon}_2\right).
\end{eqnarray}
 \end{defn}

\begin{rem} A notational warning is in order. Tangent vectors get decomposed into vertical and horizontal components
in two stages. First, we write $\upsilon=\big(\upsilon',\overrightarrow{\upsilon}\big)\in T_xX$ with respect to
the connection $\alpha$; thus $\overrightarrow{\upsilon}$ may be viewed in a natural manner as an element of $T_mM$ if $m=\pi(x)$.
Secondly, if $m\in M'$ then $\overrightarrow{\upsilon}\in T_mM$ may itself be decomposed as in (\ref{eqn:hvt}), and here 
horizontality refers to the decomposition of $\overrightarrow{\upsilon}-\overrightarrow{\upsilon}_t\in T_mM'$ with the respect to the
natural connection of principal $G$-bundle $M'\rightarrow M_0$.
\end{rem}

The scaling asymptotics in this paper will be expressed is a system of 
Heisenberg local coordinates (HLC for short) for $X$ centered at $x$ \cite{sz}, that we shall denote by:
$$
\gamma_x:(\theta,\mathbf{w})\in (-\pi,\pi)\times B_{2d}(\mathbf{0},\delta)\mapsto x+(\theta,\mathbf{w})\in X;
$$
here $B_{2d}(\mathbf{0},\delta)\subseteq \mathbb{R}^{2d}$ is the open ball centered at the origin of radius $\delta>0$.

Heisenberg local coordinate are especially suited to exhibiting the universal character of scaling asymptotics 
related to $\Pi$ (see \cite{sz} for a precise definition and discussion, and \S \ref{scn:HLC} below
for a brief recall of some salient
properties). For now, note that
$\theta$ is an angular coordinate along the circle fiber, and $\mathbf{w}$ is a local coordinate on $M$ with good
metric properties; in addition, the unitary local section of $A^\vee$ 
given in local coordinates by $\mathbf{w}\mapsto (\mathbf{w},1)$
is suitably adapted to the connection.
We shall also set $x+\mathbf{w}=:x+(0,\mathbf{w})$.

Given the choice of HLC centered at $x\in X$, there are induced unitary isomorphisms
$T_xX\cong \mathbb{R}\oplus \mathbb{R}^{2d}$ and $T_mM\cong \mathbb{R}^{2d}\cong \mathbb{C}^d$,
which will be implicit in the following; accordingly any $\upsilon\in T_xX$ will be written as a pair
$\upsilon=(\theta,\mathbf{w})\in \mathbb{R}\times \mathbb{R}^{2d}$; if
$\upsilon\in T_xX$ is small enough, we shall then write $x+\upsilon$ with this identification.
We shall also write $Q_{\mathrm{h}}^T\left(\mathbf{w}_1,\mathbf{w}_2\right)$ and 
$Q_{\mathrm{vt}}^T\left(\mathbf{w}_1,\mathbf{w}_2\right)$ for and (\ref{eqn:HORIZONTAL_quadratic_term})
and (\ref{eqn:Q_vt}), respectively, if $\upsilon_j=(\theta_j,\mathbf{w}_j)$ in local coordinates.

Under this unitary isomorphism $T_mM\cong \mathbb{R}^{2d}$, the direct sum decomposition 
(\ref{eqn:tangent_component_decomposition}) corresponds to the one
\begin{equation}
 \label{eqn:explicit_direct_sum_decomposition}
 \mathbb{R}^{2d}\cong \mathbb{R}^{2(d-e)}_{\mathrm{hor}}\oplus \mathbb{R}^{e}_{\mathrm{ver}}
 \oplus  \mathbb{R}^{e}_{\mathrm{trasv}},
\end{equation}
where we label each Euclidean summand according to the corresponding component.

We need some further ingredients
that go into the scaling asymptotics of 
$S_{\chi\cdot e^{-i\lambda(\cdot)}}^{(\varpi)}$.

\begin{defn}
 \label{defn:stabilizer}
If $\mathbf{0}\in \mathfrak{g}^\vee$ is a regular value of $\Phi$,
then $G$ acts locally freely on $M'$, hence \textit{a fortiori}
on $X'$. Thus any $m\in M'$ has finite stabilizer $G_m^M\subseteq G$, and if $x\in \pi^{-1}(m)$ its stabilizer 
is a normal subgroup $G^X_x\trianglelefteq G_m^M$. 
Since the action of $G$ on $A$ 
(and $A^\vee$) is fiberwise linear, 
$G^X_x=G^X_y$ if $x,y\in \pi^{-1}(m)$; thus we shall write
$G^X_m$ for $G^X_x$ when
$m=\pi(x)$. 

Let $\jmath_m=:\big[G_m^M:G^X_m\big]$ for $m\in M'$.
\end{defn}

\begin{defn}
\label{defn:effective_volume}
Let $G\cdot m\cong G/G^M_m\subseteq M$ be the $G$-orbit of $m\in M'$. 
The \textit{effective
volume} $V_{\mathrm{eff}}^M(m)$ at $m$
is the volume of $G\cdot m$ for the induced Riemannian structure of $M$ \cite{burns_gu}.

Similarly, if $x\in X'$ the effective volume $V_{\mathrm{eff}}^X(x)$ is the volume of $G\cdot x
\cong G/G^X_x\subseteq X$. As $V_{\mathrm{eff}}^X$ is obviously $S^1$-invariant, with a slight abuse of language
we shall view it as a function on $M'$. 
\end{defn}

\begin{rem}
 \label{rem:horizontality_effective}
Since the $G$-action on $X'$ is horizontal with respect to $\alpha$, if $m=\pi(x)\in M'$ then
the projection $G\cdot x\rightarrow G\cdot m$ is a local Riemannian isometry and a covering
of degree $\jmath_m$; therefore,
\begin{equation*}
V_{\mathrm{eff}}^X(m)=\jmath_m\cdot V_{\mathrm{eff}}^M(m)\,\,\Longrightarrow\,\,
\left|G^M_m\right|\cdot V_{\mathrm{eff}}^M(m)=\left|G^X_m\right|\cdot V_{\mathrm{eff}}^X(m).
\end{equation*}
\end{rem}

\begin{defn}
 If $x\in X'$, $m=\pi(x)$, and $\varpi$ an irreducible character of $G$, let us define
$A^T_\varpi:M'=\Phi^{-1}(\mathbf{0})\rightarrow \mathbb{R}$ 
$$
A^T_{\varpi}(x)=:2^{e/2}\,\frac{\dim(V_\varpi)}{V^X_{\mathrm{eff}}(m)}\cdot \varsigma_T(x)^{-(d+1-e/2)}.
$$
\end{defn}

For any $g\in G^M_m$, the differential 
$d_m\mu^M_g:T_mM\rightarrow T_mM$ is a unitary automorphism of $T_mM$; hence its Jacobian matrix 
$A_g$ in HLC centered at $x$ (with $m=\pi(x)$) is unitary.

In the following, $\sim$ will mean \lq has the same asymptotics as\rq.

\begin{thm}
 \label{thm:local_scaling_asymtpotics}
 Suppose that $\mathbf{0}\in \mathfrak{g}^\vee$ is a regular value of $\Phi$, and let $T$ be a 
 $G$-invariant first order self-adjoint T\"{o}plitz operator with $\varsigma_T>0$. 

For any sufficiently small $\epsilon>0$ any
good $\epsilon$-cut-off
 $\chi$, the following holds. 
 
Fix $x\in X'$ and adopt HLC on $X$ centered at $x$; set $m=:\pi(x)$. 

Then, uniformly for $\upsilon_j=(\theta_j,\mathbf{w}_j)\in T_xX$
 with $\|\upsilon_j\|\le C\,\lambda^{1/24}$, $j=1,2$, as $\lambda\rightarrow +\infty$ we have
\begin{eqnarray*}
 \lefteqn{S_{\chi\cdot e^{-i\lambda(\cdot)}}^{(\varpi)}
 \left(x+\frac{\upsilon_1}{\sqrt{\lambda}},x+\frac{\upsilon_2}{\sqrt{\lambda}}\right)}\\
& \sim& 2\pi\cdot A^T_{\varpi}(x)\cdot 
e^{i\sqrt{\lambda}\,(\theta_1-\theta_2)/\varsigma_T(x)+Q_{\mathrm{tv}}^T(\mathbf{w}_1,\mathbf{w}_2)}
\cdot
\left(\frac{\lambda}{\pi}\right)^{d-e/2}\nonumber\\
&&\cdot \frac{1}{\big|G^X_m\big|}\,
\sum_{g\in G^X_m}\,\overline{\chi_\varpi(g)}\cdot
 e^{Q^T_{\mathrm{h}}\big(\mathbf{w}_{1},A_g\mathbf{w}_{2}\big)}
\cdot S_g(\lambda,x,\upsilon_1,\upsilon_2), 
 \end{eqnarray*}
where each factor $S_g(\lambda,x,\upsilon_1,\upsilon_2)$ 
satisfies an asymptotic expansion of the form
\begin{equation*}
 S_g(\lambda,x,\upsilon_1,\upsilon_2)\sim 1+\sum_{l\ge 1}\lambda^{-l/2}\,F_{gl}(x,\upsilon_1,\upsilon_2),
\end{equation*}
$F_{gl}(x,\upsilon_1,\upsilon_2)$ being a polynomial in $\upsilon_a=(\theta_a,\mathbf{w}_a)$, $a=1,2$, of total degree
$\le 11\,l$ (also depending on $T$). 
\end{thm}

Let us consider the special case where $\upsilon_1=\upsilon_2=(0,\mathbf{w})$,
where $\mathbf{w}=\mathbf{w}_\mathrm{t}\in N_m=T_mM_{\mathrm{trasv}}\cong \mathbb{R}^{e}_{\mathrm{trasv}}$
in
the notation of (\ref{eqn:tangent_component_decomposition})
and (\ref{eqn:explicit_direct_sum_decomposition}). 

\begin{defn}
 \label{defn:a_varpi_Phi}
Let us define a function $a_{\Phi,\varpi}:M'\rightarrow \mathbb{R}$ by setting
\begin{equation}
 \label{eqn:character_L2_product}
a_{\Phi,\varpi}(m)=:\frac{1}{\big|G^X_m\big|}\cdot
\sum_{g\in G^X_m}\,\chi_\varpi(g)=\left\langle \left.\chi_\varpi\right|_{G^X_m},1\right\rangle_{L^2(G^X_m)}.
\end{equation}
Note that $a_{\Phi,\varpi}$ is real, as by unitarity $\overline{\chi_\varpi(g)}=\chi_\varpi\big(g^{-1}\big)$,
for any $g\in G$. Also,
there is a dense open subset of $M''\subseteq M'$ on which the conjugacy class of $G^X_m$ is constant
\cite{ggk},
hence $a_{\Phi,\varpi}$ is constant on $M''$;
we shall denote by $a_{\mathrm{gen}}(\Phi,\varpi)\in \mathbb{R}$ the constant value it takes on $M''$.
\end{defn}

For example, if $\widetilde{\mu}$ is generically free on $X'$, then $a_{\mathrm{gen}}(\Phi,\varpi)=\dim(V_\varpi)$.

\begin{cor}
\label{cor:local_weyl_law}
 Under the same assumptions of the Theorem, uniformly in $\mathbf{w}=\mathbf{w}_t\in N_m$
with $\|\mathbf{w}\|\le C\,\lambda^{1/24}$ we have
\begin{eqnarray*}
 \lefteqn{S_{\chi\cdot e^{-i\lambda(\cdot)}}^{(\varpi)}
 \left(x+\frac{\mathbf{w}}{\sqrt{\lambda}},x+\frac{\mathbf{w}}{\sqrt{\lambda}}\right)}\\
& \sim& 2\pi\cdot A^T_{\varpi}(x)\,a_{\Phi,\varpi}(m)\,
\left(\frac{\lambda}{\pi}\right)^{d-e/2}\cdot  
\exp\left(-\frac{2}{\varsigma_T(x)}\,\|\mathbf{w}\|^2\right)\\
&&\cdot\left[1+\sum_{j\ge 1}\lambda^{-j/2}\,F_j(x,\mathbf{w})\right], 
 \end{eqnarray*}
where $F_j$ is a polynomial in $\mathbf{w}$ of degree $\le 11j$, of the same parity as $j$
(also depending on $T$).
\end{cor}

The Theorem yields information about the asymptotic concentration behavior of the equivariant eigenfunctions of $T$;
from this, one obtains by integration an estimate on the asymptotic distribution of the eigenvalues $\lambda_j^{(\varpi)}$.
Let us define:
\begin{equation}
 \label{eqn:Gamma_Phi_T}
\Gamma(\Phi,\varsigma_T)=:\int_{X'}\frac{1}{V^X_{\mathrm{eff}}(m)}\,\varsigma_T(x)^{-(d-e+1)}\,\mathrm{d}V_{X'}(x),
\end{equation}
where $m=\pi(x)$, and
$\mathrm{d}V_{X'}$ is the induced volume form on $X'$. 
\begin{cor}
\label{cor:trace_asymptotics}
 Under the same assumptions, there is an asymptotic expansion
\begin{eqnarray*}
 \sum_j\widehat{\chi}\big(\lambda-\lambda_j^{(\varpi)}\big)&\sim& 2\pi\,\dim(V_\varpi)\,a_{\mathrm{gen}}(\Phi,\varpi)
\,\Gamma (\Phi,\varpi)\,
\left(\frac{\lambda}{\pi}\right)^{d-e}\\
&&\cdot\left[1+\sum_{j\ge 0}\lambda^{-j}\,E_j\right],
\end{eqnarray*}

\end{cor}

The estimate in Corollary \ref{cor:trace_asymptotics} leads by a classical Tauberian argument
to an estimate on the $\varpi$-equivariant counting function of $T$, defined as 

\begin{equation}
 \label{eqn:counting_function}
N_T^{(\varpi)}(\lambda)=:\sharp\left\{j:\,\lambda_j^{(\varpi)}\le \lambda\right\}.
\end{equation}

\begin{cor}
 \label{cor:weyl_law}
As $\lambda\rightarrow+\infty$, one has
$$
N_T^{(\varpi)}(\lambda)=
\frac{\pi}{d-e+1}\cdot \dim(V_\varpi)\,a_{\mathrm{gen}}(\Phi,\varpi)\,\Gamma (\Phi,\varpi)\, 
\left(\frac{\lambda}{\pi}\right)^{d-e+1}
+O\left(\lambda^{d-e}\right).
$$
\end{cor}

In the spirit of \cite{hor} and \cite{bdm-g}, the leading asymptotics of
$N_T^{(\varpi)}(\lambda)$ may be related to an appropriate symplectic volume. 

Just to fix ideas, let us
make the symplifying assumption that $\mu^M$ is free on $M'$, leaving it to the interested reader to
consider the general case. Then the symplectic quotient
$\widehat{M}=:M'/G$, with the induced symplectic structure $\widehat{\omega}$, 
is a Hodge manifold in a natural manner \cite{guillemin-sternberg}. Furthermore, $\widehat{X}=:X'/G$
is the unit circle bundle on $\widehat{M}$ for the positive line bundle $\widehat{A}$
induced by $A$ on $\widehat{M}$ by passage to the quotient. Let $\widehat{\alpha}$ be the connection form of the latter,
and denote by $\widehat{\Sigma}$ the cone in $T^*\widehat{X}$ sprayed by $\widehat{\alpha}$.
Then $\widehat{\Sigma}=\Sigma'/G$, where $\Sigma'$ is the restriction of $\Sigma$ to $X'$.

In addition, being $\widetilde{\mu}$-invariant, $\varsigma_T$ descends to a $\mathcal{C}^\infty$
function on $\widehat{X}$, and similarly $\sigma_T$ descends to a $\mathcal{C}^\infty$ homogeneous
function $\widehat{\sigma}_T$ on $\widehat{\Sigma}$. 

\begin{cor}
 \label{cor:weyl_law_volume}
Let $\widehat{\Sigma}_1\subseteq \widehat{\Sigma}$
be the locus where $\widehat{\sigma}_T\le 1$.  
Then, under the previous assumptions, as $\lambda\rightarrow+\infty$, one has
$$
N_T^{(\varpi)}(\lambda)=
\dim(V_\varpi)^2\,\mathrm{vol}\left(\widehat{\Sigma}_1\right)\, 
\left(\frac{\lambda}{2\,\pi}\right)^{d-e+1}
+O\left(\lambda^{d-e}\right),
$$
\end{cor}

As in the case of \cite{p_weyl}, the arguments in this paper combine the classical approach to
trace formulae and Weyl laws for pseudodifferential operators \cite{dg}, \cite{hor}, \cite{gr_sj}
with the microlocal theory of the Szeg\"{o} kernel \cite{f}, \cite{bdm_sj}, and especially its
description as an FIO with complex phase function in the latter article. This follows the philosophy
in \cite{sz} and \cite{zel_tian}, where the theory of \cite{bdm_sj} is specialized to the case of
of algebro-geometric Szeg\"{o} kernels, and in fact we shall extensively build on ideas ad techniques from the
latter papers. In order to include symmetries                
in this picture,
and describe how the local contribution to the equivariant trace formula
asymptotically concentrates near the zero locus of the moment map, we shall furthermore adapt the approach 
and techniques in \cite{p_JSG}, \cite{pao_IJM}, \cite{pao_JMP}.

\section{Preliminaries}

\subsection{Smoothings of wave operators}

We collect here some well-known basic facts about smoothed averages of wave operators of the
form $e^{i\tau Q}$, and their T\"{o}plitz counterparts (\cite{gr_sj}, \cite{bdm-g})

Let $T$ be a $G$-invariant first order self-adjoint T\"{o}plitz operator with $\varsigma_T>0$.
By the theory of \cite{bdm-g}, there exists a $G$-invariant first order elliptic self-adjoint
pseudo-differential operator of classical type $Q$ on $X$, such that 
\begin{equation}
 \label{eqn:special_Q}
 [\Pi,Q]=0,\,\,\,\,\,\sigma_Q>0,\,\,\,\,\,T=\Pi\circ Q\circ \Pi.
\end{equation}
In fact, such a $Q$ exists by Lemma 12.1 of \cite{bdm-g}, 
and averaging over $G$ yields invariance. In particular, $T$ is the restriction of $Q$ to $H(X)$.

Let $U(\tau)=:e^{i\tau Q}$, $U_T(\tau)=:e^{i\tau T}$; thus 
$U_T(\tau)=\Pi\circ U(\tau)\circ \Pi=U(\tau)\circ \Pi$.
Let $\beta_1\le \beta_2\le \cdots$ be the eigenvalues of $Q$, repeated according to multpicity,
and let $(f_j)$ be a complete orthonormal system of $L^2(X)$ with $Q(f_j)=\beta_j\,f_j$; then
the same holds of $\left(f_j^g\right)$ for any $g\in G$, where $f^g=f\circ \mu^X_{g^{-1}}$. 
It follows that the
distributional kernel $U(\tau)\in \mathcal{D}'(X\times X)$ satisfies
\begin{eqnarray*}
\label{eqn:U_tau_distrib}
U(\tau)(x,y)&=&U(\tau)\left(\mu^X_g(x),\mu^X_g(y)\right)
\end{eqnarray*}
for any $g\in G$.
The same considerations hold with $U_T(\tau)$ in place of $U(\tau)$.

If $\chi\in \mathcal{S}(\mathbb{R})$, the averaged operator 
\begin{equation}
 \label{eqn:smoothed_wave_operator}
U_\chi=:\int_{-\infty}^{+\infty}\chi (\tau)\,U(\tau)\,\mathrm{d}\tau
\end{equation}
is $\mathcal{C}^\infty$, with Schwartz kernel the 
series
$$
U_\chi(x,y)=\sum_j\widehat{\chi}(-\beta_j)\,f_j(x)\cdot\overline{f_j(y)},
$$
which converges uniformly and absolutely in $\mathcal{C}^\infty(X\times X)$ \cite{gr_sj}. The sequences 
$(\lambda_j)$'s and $\big(\lambda_j^{(\varpi)}\big)$
in (\ref{eqn:average_wave_operator}) and (\ref{eqn:equivariant_average_wave_operator})
are subsequences of $(\beta_j)$, and we may assume without loss that the same holds for the corresponding eigenfunctions. 
Similar conclusions
therefore hold for $S_{\chi}$ in (\ref{eqn:average_wave_operator}) and (\ref{eqn:kernel_average_wave_operator}),
and its equivariant drifting counterpart (\ref{eqn:traveling_average_wave_operator}) and (\ref{eqn:equivariant_average_wave_operator}).

Furthermore, $U(\tau)$ is an FIO associated to the Hamiltonian flow $\phi_\tau^{T^*X}$ of $\sigma_Q$ on $T^*X$
\cite{dg}, \cite{gr_sj}.
More precisely, 
for $\tau\sim 0$ and in the neighborhood of the diagonal in $X\times X$
we have $U(\tau)=V(\tau)+R(\tau)$, where $R(\tau)$ is a smoothing operator, while in local coordinates
$V(\tau)$ has the form
\begin{equation}
 \label{eqn:microlocal_V_tau}
V(\tau)(x,y)=\dfrac{1}{(2\pi)^{2d+1}}\,\int_{\mathbb{R}^{2d+1}}\,e^{i[\varphi(\tau,x,\eta)-\langle y,\eta\rangle]}
\,a(\tau,x,y,\eta)\,\mathrm{d}\eta,
\end{equation}
where
the phase and amplitude are as follows. First, $\varphi(\tau,\cdot,\cdot)$ is the generating function of $\phi_\tau^{T^*X}$,
and therefore it satisfies the Hamilton-Jacobi equation. Since $\phi_0^{T^*X}$ is the identity, 
for $\tau\sim 0$ we have
\begin{equation}
 \label{eqn:Hamilton_Jacobi}
\varphi(\tau,x,\eta)=\langle x,\eta\rangle+\tau\,\sigma_Q(x,\eta)+\|\eta\|\,O\left(\tau^2\right).
\end{equation}
On the other hand the amplitude is a classical symbol $a(\tau,\cdot,\cdot,\cdot)\in S^0_{\mathrm{cl}}$,
and $a(0,\cdot,\cdot,\cdot)=1/\mathcal{V}(y)$, where $\mathcal{V}(x)\,\mathrm{d}x$ is the local coordinate
expression of $\mathrm{d}V_X$ (see also the discussion in \cite{p_weyl}). 

\subsubsection{Cut-offs}
\label{sct:cut_offs}

We shall choose the cut-off function $\chi$ as follows.

Given $\epsilon>0$, choose $\gamma\in \mathcal{C}^\infty\big((-\epsilon/2,\epsilon/2)\big)$
real, simmetric and non-negative; in particular its Fourier transform $\widehat{\gamma}$
is real. After normalization, 
we may assume $\|\gamma\|_{L^2}=1$.
Setting $\chi=\gamma *\gamma$ (convolution) we obtain the following
well-known:

\begin{lem}
\label{lem:test_function}
 For any $\epsilon>0$ there exists $\chi\in \mathcal{C}^\infty_0\big((-\epsilon,\epsilon)\big)$ such that
 $\chi\ge 0$, $\chi(0)=1$, $\widehat{\chi}\ge 0$.
\end{lem}

With this choice, that $U_\chi$ is smoothing may also be seen using (\ref{eqn:microlocal_V_tau}) and (\ref{eqn:Hamilton_Jacobi})
to integrate by parts in $\mathrm{d}\tau$ in (\ref{eqn:smoothed_wave_operator}).

\subsection{Heisenberg local coordinates}
\label{scn:HLC} 
Let $\gamma_x\big((\theta,\mathbf{w}))=x+(\theta,\mathbf{w})$ be a system of Heisenberg local coordinates
on $X$ centered at $x$ \cite{sz}. Then the following holds:

\begin{enumerate}
 \item The standard circle action $r:S^1\times X\rightarrow X$ is expressed by a translation in the angular coordinate: where defined,
$$
r_{e^{i\vartheta}}\big(x+(\theta,\mathbf{w})\big)=x+(\vartheta+\theta,\mathbf{w}).
$$
\item If $m=\pi(x)$ let us set $m+\mathbf{w}=:\pi\big(x+(0,\mathbf{w})\big)$; then 
$\mathbf{w}\in  B_{2d}(\mathbf{0},\delta)\mapsto m+\mathbf{w}$ is a local coordinate chart centered at $m$, inducing
a unitary isomorphism $\mathbb{C}^d\cong T_mM$; in other words, $\omega_m$ and $J_m$ correspond to the standard complex and symplectic
structures on $\mathbb{R}^{2d}\cong \mathbb{C}^d$.
\item  
$\gamma_x$ induces at $x$ an isomorphism $\mathbb{R}\oplus \mathbb{R}^{2d}\cong T_xX$ compatible with the direct sum decomposition
$
T_xX\cong \mathcal{V}_x\oplus \mathcal{H}_x
$, that is, $$\mathcal{V}_x\cong \mathbb{R}\oplus(\mathbf{0}),\,\,\,\,\,\,\,\mathcal{H}_x\cong (0)\oplus \mathbb{R}^{2d}.$$
\item If we write $\mathbf{w}\in \mathbb{R}^{2d}$ as a complex vector $\mathbf{z}\in \mathbb{C}^{d}$, the local coordinate expression of
of $\alpha$ at $x+(\theta,\mathbf{z})$ is 
$$
\alpha=\mathrm{d}\theta+\sum_{j=1}^d \left(A_j(\mathbf{z},\overline{\mathbf{z}})\,\mathrm{d}\mathbf{z}_j+
\overline{A_j(\mathbf{z},\overline{\mathbf{z}})}\,\mathrm{d}\overline{\mathbf{z}}_j\right),
$$
where $A_j=-(i/2)\,\overline{z}_j+O\left(\|\mathbf{z}\|^2\right)$.
\end{enumerate}

\subsection{The Szeg\"{o} kernel}

The Szeg\"{o} projector of $X$ is the orthogonal projector
$\Pi:L^2(X)\rightarrow H(X)$. 
Its distributional kernel, the Szeg\"{o} kernel $\Pi\in \mathcal{D}'(X\times X)$, 
has singular support along the diagonal \cite{f}. The following analysis
is based on the description of $\Pi$ 
as an FIO with a complex
phase of positive type in \cite{bdm_sj}: up to smoothing terms,
\begin{equation}
 \label{eqn:microlocal_Pi}
 \Pi(x,y)=\int_0^{+\infty}e^{it\psi(x,y)}\,s(x,y,t)\,\mathrm{d}t.
\end{equation}
Here $\psi$ is essentially determined along the diagonal by the metric, and the amplitude is a classical symbol
admitting an asymptotic expansion of the form
$$
s(x,y,t)\sim \sum_{j\ge 0}t^{d-j}\,s_{j}(x,y)
$$
(see also the discussion in \S 2 of \cite{sz}). 
It follows from (\ref{eqn:microlocal_Pi}) that the wave front of $\Pi$ 
is the anti-diagonal in $\Sigma\times \Sigma$:
\begin{equation}
 \label{eqn:wave_front_Pi}
\mathrm{WF}(\Pi)=\Big\{\big((x,r\,\alpha_x),\,(x,-r\,\alpha_x)\big):\,x\in X,\,r>0\Big\}
\end{equation}
(\cite{bdm_sj}, \cite{bdm-g}).
In a system of Heisenberg local coordinates centered at $x\in X$, by the discussion in \S 3 of
\cite{sz} we have
\begin{eqnarray}
\label{eqn:expansion_sz}
 \lefteqn{t\,\psi\big(x+(\theta,\mathbf{v}),x+(\theta',\mathbf{v}')\big)}\\
&=&it\,\left[1-e^{i(\theta-\theta')}\right]-it\,\psi_2\left(\mathbf{v},\mathbf{v}'\right)
\,e^{i(\theta-\theta')}+t\,R_3\left(\mathbf{v},\mathbf{v}'\right)\,e^{i(\theta-\theta')}.\nonumber
\end{eqnarray}
Furthermore, $s_0(x,x)=\pi^{-d}$.

\subsection{The underlying dynamics on $X$}

The arguments in this paper are based on the local analysis of the Fourier T\"{o}plitz
operator $U_T(\tau)$. In the $S^1$-invariant case, $\varsigma_T$
is (the pull-back to $X$ of) a $\mathcal{C}^\infty$ function on $M$, and
$U_T(\tau)$ may be regarded as a quantization of the classical dynamics of $\varsigma_T$.
Before embarking on the actual proof, it is in order to put things in perspective by
showing that even in this more general case
there is a \lq conformally contact\rq\, dynamics in the picture, which is \lq quantized\rq\, by $U_T(\tau)$.
Before doing so, however, let us briefly clarify the relation between the flows of
$\sigma_Q$ on $T^*X$ and of $\sigma_T$ on $\Sigma$. 

\subsubsection{The flows of $\sigma_Q$ and $\sigma_T$}

Recall that among the principal symbols $\sigma_T:\Sigma\rightarrow \mathbb{R}$ of $T$ and
$\sigma_Q:(T^*X)_0\rightarrow \mathbb{R}$ of $Q$, where $(T^*X)_0\subseteq T^*X$ 
is the complement of the zero section, and the reduced symbol
$\varsigma_T:X\rightarrow \mathbb{R}$ there are the relations
\begin{equation}
 \label{eqn:symbol_T_Q}
 \sigma_Q\big(x,r\,\alpha_x\big)=\sigma_T\big(x,r\,\alpha_x\big)=r\,\varsigma_T(x).
\end{equation}
Let $\phi_\tau^{T^*X}:(T^*X)_0\rightarrow (T^*X)_0$ and 
$\phi_\tau^{\Sigma}:\Sigma\rightarrow \Sigma$
be the Hamiltonian flows generated by $\sigma_Q$ on $(T^*X)_0$ and 
by $\sigma_T$ on the symplectic submanifold $\Sigma$.

\begin{lem}
\label{lem:invariant_Sigma}
 Given that $[\Pi,Q]=0$, for every $\tau\in \mathbb{R}$ we have
 $
 \phi_\tau^{T^*X}(\Sigma)=\Sigma.
 $
\end{lem}

\begin{proof}
 Since $[\Pi,U(\tau)]=0$, wave fronts satisfy 
 \begin{eqnarray}
\label{eqn:commuting_wave_fronts}
\lefteqn{\mathrm{WF}'(\Pi)\circ\mathrm{WF}'\big(U(\tau)\big)=\mathrm{WF}'\big(\Pi\circ U(\tau)\big)}\\
&=&\mathrm{WF}'\big(U(\tau)\circ \Pi\big)=\mathrm{WF}'\big(U(\tau)\big)\circ \mathrm{WF}'(\Pi).
\nonumber
 \end{eqnarray}
Now by (\ref{eqn:wave_front_Pi}) and \cite{dg}
\begin{eqnarray}
 \label{eqn:wave_front_Pi'}
 \mathrm{WF}'(\Pi)&=&\Big\{\big((x,r\,\alpha_x),(x,r\,\alpha_x)\big)\,:\,x\in X,\,r>0\Big\},\\
\label{eqn:wave_front_tau}
\mathrm{WF}'\big(U(\tau)\big)&=&\mathrm{graph}\left(\phi_{-\tau}^{T^*X}\right)\\
&=&\left\{\left(\phi_{\tau}^{T^*X}(x,\eta),(x,\eta)\right):\,(x,\eta)\in (T^*X)_0\right\}.\nonumber
\end{eqnarray}

By (\ref{eqn:commuting_wave_fronts}) and (\ref{eqn:wave_front_Pi'}) - (\ref{eqn:wave_front_tau})
we obtain 
\begin{eqnarray}
 \label{eqn:equality_wave_fronts}
\lefteqn{\left\{\left(\phi_{\tau}^{T^*X}(x,r\,\alpha_x),\,(x,r\,\alpha_x)\right):\,x\in X,\,r>0\right\}}\\
&=&\left\{\left((x,r\,\alpha_x),\,\phi_{-\tau}^{T^*X}(x,r\,\alpha_x)\right):\,x\in X,\,r>0\right\}.\nonumber
\end{eqnarray}

Clearly (\ref{eqn:equality_wave_fronts}) implies the statement. 
\end{proof}

\begin{cor}
\label{cor:restriction_of_flows}
Given that $[\Pi,Q]=0$, $\phi_\tau^{\Sigma}$ is the restriction 
of $\phi_\tau^{T^*X}$ to $\Sigma$.
\end{cor}

\subsubsection{The contact flow on $X$}

In the $S^1$-invariant case, $U_T(\tau)$ is a quantization 
of a Hamiltonian flow on $M$. Namely, let $f=\varsigma_T$, naturally
interpreted as a $\mathcal{C}^\infty$ real function on $M$, and let 
$\phi_\tau^M:M\rightarrow M$ be the flow of $f$ generated by the
Hamiltonian vector field $\upsilon_f\in\mathfrak{X}(M)$ of $f$ with respect
to $(M,2\,\omega)$. Then
$$
\widetilde{\upsilon}_f=:\upsilon_f^\sharp -f\,\frac{\partial}{\partial\theta}
$$
is a contact vector field on $X$ lifting $\upsilon_f$, and $U_T(\tau)$ is 
a T\"{o}plitz Fourier operator associated to the corresponding contact flow $\phi^X_\tau$.

By way of motivation, we shall show here that even in the general case
$U_T(\tau)$ is associated to a suitable
underlying conformally contact dynamics on $X$.

Given the direct sum decomposition $TX\cong \mathcal{V}\oplus \mathcal{H}$, we also have
$T^*X\cong \mathcal{V}^*\oplus \mathcal{H}^*$, where $\mathcal{V}^*=\mathcal{H}^0=\mathrm{span}(\alpha)$,
$\mathcal{H}^*=\mathcal{V}^0$ (annihilators). 
We shall accordingly write
\begin{eqnarray}
 \mathrm{d}\varsigma_T=\mathrm{d}^\mathrm{v}\varsigma_T+\mathrm{d}^\mathrm{h}\varsigma_T&\mathrm{with}&
 \mathrm{d}^\mathrm{v}\varsigma_T\in \mathcal{C}^\infty(X,\mathcal{V}),\,
 \mathrm{d}^\mathrm{h}\varsigma_T\in \mathcal{C}^\infty(X,\mathcal{H}).
\end{eqnarray}

Since $\partial/\partial\theta$ and $\alpha$ are in duality, we have 
$\mathrm{d}^\mathrm{v}\varsigma_T=(\partial_\theta\varsigma_T)\cdot \alpha$. On the other hand, the pull back
$\omega=\pi^*(\omega)$ is a symplectic structure on the vector sub-bundle
$\mathcal{H}\subseteq T^*X$ (here and in the following we shall more or less systematically omit symbols of 
pull-back in order to simplify notation). Let $\mathcal{X}^\mathrm{h} (X)$ be the space of horizontal vector fields on $X$,
that is, smooth sections of $\mathcal{H}$.
Then $\mathrm{d}^\mathrm{h}\varsigma_T$ corresponds to a unique 
$\upsilon_T^\mathrm{h}\in \mathfrak{X}^\mathrm{h}(X)$ under $2\,\omega$

Clearly, $\Sigma\stackrel{q}{\rightarrow}X$ is a trivial $\mathbb{R}^+$-bundle.
In terms of the diffeomorphism $(x,r\,\alpha_x)\in\Sigma\mapsto (x,r)\in X\times \mathbb{R}^+$ and the decomposition
$TX\cong \mathcal{V}\oplus \mathcal{H}$, omitting the pull-back symbol $q^*$ we have
\begin{eqnarray}
 T\Sigma&\cong &TX\oplus \mathrm{span}\,\left\{\frac{\partial}{\partial r}\right\}\\
& \cong& \mathcal{V}\oplus \mathcal{H}\oplus \mathrm{span}\,\left\{\frac{\partial}{\partial r}\right\}
\cong \mathcal{H}\oplus \mathrm{span}\,\left\{\frac{\partial}{\partial \theta},\,\frac{\partial}{\partial r}\right\}.
\nonumber
\end{eqnarray}
Let $\mathfrak{X}^\mathrm{h}(\Sigma)\subseteq\mathfrak{X}(\Sigma)$ 
be the space of vector fields on $\Sigma$ tangent to the distribution
$\mathcal{H}=q^*(\mathcal{H})$.
Given $\upsilon\in \mathfrak{X}(\Sigma)$, one can accordingly write
$$
\upsilon=\upsilon^h+a\,\frac{\partial}{\partial \theta}+b\,\frac{\partial}{\partial r}
$$
for unique $\upsilon^\mathrm{h}\in \mathfrak{X}^\mathrm{h}(\Sigma)$ and $a,b\in \mathcal{C}^\infty(\Sigma)$.

Let $\omega_\Sigma$ the symplectic structure on $\Sigma$, given by
\begin{equation}
 \label{eqn:symplectic_Sigma}
 \omega_\Sigma=\mathrm{d}(r\,\alpha)=\mathrm{d}r\wedge \alpha+r\,\mathrm{d}\alpha=
\mathrm{d}r\wedge \alpha+2r\,\omega,
\end{equation}
where $\omega=q^*(\omega)$. Let $\upsilon_T\in \mathfrak{X}(\Sigma)$ be the Hamiltonian vector field of $\sigma_T$
with respect to $\omega_\Sigma$.

We have
\begin{eqnarray}
 \label{eqn:differential_sigmaT}
 \mathrm{d}\sigma_T&=&\mathrm{d}(r\,\varsigma_T)=\varsigma_T\,\mathrm{d}r+r\,\mathrm{d}\varsigma_T\\
 &=&\varsigma_T\,\mathrm{d}r+r\,\mathrm{d}^{\mathrm{v}}\varsigma_T+r\,\mathrm{d}^{\mathrm{h}}\varsigma_T=
 \varsigma_T\,\mathrm{d}r+r\,(\partial_\theta\varsigma_T)\cdot \alpha+r\,\mathrm{d}^{\mathrm{h}}\varsigma_T.
 \nonumber
\end{eqnarray}

Let $\upsilon_T^\mathrm{h}\in \mathfrak{X}^{\mathrm{h}}(X)$ be the horizontal vector field on $X$ that
 corresponds to $\mathrm{d}^\mathrm{h}\varsigma_T$ under $2\,\omega$ (pulled back to $\mathcal{H}$):
 $$
 \mathrm{d}^\mathrm{h}\varsigma_T=2\,\iota\left(\upsilon_T^\mathrm{h}\right)\,\omega=
 2\,\omega\left(\upsilon_T^\mathrm{h},\cdot\right)
 $$
and let us define
$$
\upsilon_T^X=:\upsilon_T^\mathrm{h}-\varsigma_T\,\frac{\partial}{\partial\theta}\in \mathfrak{X}(X),
$$
and view $\upsilon_T^X$ as a vector field on $\Sigma$ in a natural manner.
Inspection of (\ref{eqn:differential_sigmaT}) then shows the following:

\begin{lem}
The Hamiltonian vector field of $\sigma_T$ with respect to $\omega_\Sigma$ is given by
 $$
 \upsilon_T=\upsilon_T^\mathrm{h}-\varsigma_T\,\frac{\partial}{\partial\theta}
 +r\,(\partial_\theta\varsigma_T)\,\frac{\partial}{\partial r}=\upsilon_T^X+
 r\,(\partial_\theta\varsigma_T)\,\frac{\partial}{\partial r}.
 $$
\end{lem}

\begin{cor}
 \label{cor:hamiltonian_lift}
The Hamiltonian flow $\phi_\tau^\Sigma:\Sigma\rightarrow \Sigma$
of $\sigma_T$ is a lifting the flow $\phi^X_\tau:X\rightarrow X$ of $\upsilon_T^X$.
\end{cor}

As we shall see presently, in general the flow $\phi^X_\tau$ generated by $\upsilon_T^X$ does not preserve
$\alpha$, and $\varsigma_T$ is not a constant of motion; nonetheless, $\phi^X_\tau$ always 
leaves $\Sigma$ invariant.

\begin{lem}
 \begin{enumerate}
  \item The Lie derivative of $\alpha$ with respect to $\upsilon_T^X$ is given by
  $$
  L_{\upsilon_T^X}(\alpha)=-(\partial_\theta\varsigma_T)\cdot\alpha.
  $$
  \item The derivative of $\varsigma_T$ along $\phi^X_\tau$ is 
  $
  \upsilon_T^X(\varsigma_T)=-\varsigma_T\cdot (\partial_\theta\varsigma_T).
  $
\item The 1-form 
$(1/\varsigma_T)\cdot \alpha$
is a contact form on $X$, and is $\phi^X_\tau$-invariant.
 \end{enumerate}

\end{lem}

\begin{proof}
 Regarding the first point, 
 \begin{eqnarray*}
   L_{\upsilon_T^X}(\alpha)&=&\mathrm{d}\left(\iota \big(\upsilon_f^X\big)\,\alpha\right)+
   \iota \big(\upsilon_T^X\big)\,\mathrm{d}\alpha\\
   &=&-\mathrm{d}\varsigma_T+\iota\left(\upsilon_T^\mathrm{h}\right)\,2\omega=
   -\mathrm{d}^\mathrm{v}\varsigma_T=-(\partial_\theta\varsigma_T)\cdot\alpha.
  \end{eqnarray*}
  Next, 
  \begin{eqnarray*}
   \upsilon_T^X(\varsigma_T)&=&\upsilon_T^\mathrm{h}(\varsigma_T)-\varsigma_T\,\frac{\partial}{\partial\theta}\varsigma_T
   =\mathrm{d}\varsigma_T\left(\upsilon_T^\mathrm{h}\right)-\varsigma_T\,\frac{\partial}{\partial\theta}\varsigma_T=
   \mathrm{d}^\mathrm{h}\varsigma_T\left(\upsilon_T^\mathrm{h}\right)-\varsigma_T\,\frac{\partial}{\partial\theta}\varsigma_T\\
   &=&2\,\omega\left(\upsilon_T^\mathrm{h},\upsilon_T^\mathrm{h}\right)-
   \varsigma_T\,\frac{\partial}{\partial\theta}\varsigma_T=-\varsigma_T\cdot (\partial_\theta\varsigma_T).
  \end{eqnarray*}
Finally,
\begin{eqnarray*}
 L_{\upsilon_T^X}\left(\frac{1}{\varsigma_T}\cdot \alpha\right)&=&-\frac{1}{\varsigma_T^2}\,\upsilon_T^X(\varsigma_T)\,\alpha
 +\frac{1}{\varsigma_T}\cdot L_{\upsilon_T^X}\left( \alpha\right)\\
 &=&-\frac{1}{\varsigma_T^2}\,\left(-\varsigma_T\cdot (\partial_\theta\varsigma_T)\right)\cdot \alpha
 +\frac{1}{\varsigma_T}\cdot\big( -(\partial_\theta\varsigma_T)\big)\cdot\alpha=0
\end{eqnarray*}

\end{proof}

\begin{cor}
 For any $t\in \mathbb{R}$, we have
 $$
 \left(\phi^X_\tau\right)^*(\alpha)=\frac{\varsigma_T\circ \phi_\tau^X}{\varsigma_T}\cdot \alpha.
 $$
 In particular,
 the cotangent lift of $\phi_\tau^X$ leaves $\Sigma$ invariant.
\end{cor}

On the other hand, the cotangent lift of $\phi^X_\tau$ is Hamiltonian, generated by the Hamiltonian
function 
$H=:-\lambda_{\mathrm{can}}\big(\upsilon_T^X\big)$, where $\lambda_{\mathrm{can}}$ is the tautological
1-form on $T^*X$; on $\Sigma$, 
$$
H\big((x,\,r\,\alpha_x)\big)=-r\,\alpha_x\left(\upsilon_T^X(x)\right)
=r\,\varsigma_T(x)=\sigma_T\big((x,\,r\,\alpha_x)\big).
$$

\begin{cor}
 The flow $\phi^\Sigma_\tau$ coincides with the restriction to $\Sigma\subseteq T^*X$
of the cotangent lift of $\phi^X_\tau$.
\end{cor}

\subsection{Some linear algebra}

Let us collect here some facts from linear algebra that will be handy in the following.

\begin{lem}
 \label{lem:signature_1}
Let $R,\,S$ be $r\times r$ matrices, with $S$ symmetric, and consider the symmetric 
$2r\times 2r$ symmetric matrix
$$
C=C(R,S)=:\begin{pmatrix}
           0&R^T\\
R&S
          \end{pmatrix}.
$$
Then
$
\det(C)=(-1)^r\,\det(R)^2
$. If furthermore $\det(R)>0$ then the signature of $C$ is $\mathrm{sgn}(C)=0$.
\end{lem}

\begin{proof}
 The first statement is a straightforward computation by row operations. As to the second, let us remark that
the signature (actually, the number of positive and negative eigenvalues) 
is locally constant on the space of non-degenerate symmetric matrices. If $\det (R)>0$, we may find
a continuous path $R_t$, $0\le t\le 1$, with $R_0=R$ and $R_1=I_{r}$ (the identity matrix), and $\det(R_t)>0$
for every $t$. Thus
$$
C_t=:\begin{pmatrix}
      0&R_t^T\\
R_t&(1-t)\,S
     \end{pmatrix}
$$
is a family of non-degenerate symmetric matrices, with $C_0=C$, and
one easily checks that $\mathrm{sgn}(C_1)=0$.
\end{proof}

\subsection{Effective volumes and matrices}
\label{sctn:effective_volumes}

For $\xi \in \mathfrak{g}$, let $\xi_M\in \mathfrak{X}(M)$ be the vector field it induces on $M$ under $\mu$, and
for a given $m\in M$ 
let $\mathfrak{g}_M(m)\subseteq T_mM$ be the vector subspace of all the $\xi_M(m)$ ($\xi\in \mathfrak{g}$).
If $\mathbf{0}\in \mathfrak{g}^\vee$ is a regular value of $\Phi$, then $G$ acts on $M'=\Phi^{-1}(\mathbf{0})$ locally freely.
If $m\in M'$, therefore, the evaluation map $\mathrm{val}_m:\mathfrak{g}\rightarrow T_mM$ is injective, and a linear
isomorphism $\mathfrak{g}\rightarrow \mathfrak{g}_M(m)$. 
Let $\mathcal{B}=(\xi_j)$ be a fixed orthonormal basis of $\mathfrak{g}$, and $\mathcal{B}_m=(\mathbf{u}_j)$ 
be an orthonormal basis of $\mathfrak{g}_M(m)
\subseteq T_mM$ (with the restricted metric), and let $C_m=M^\mathcal{B}_{\mathcal{B}_m}(\mathrm{val}_m)$ 
be the $d\times d$ matrix representing $\mathrm{val}_m:\mathfrak{g}\rightarrow \mathfrak{g}_M(m)$
with respect to these two basis. Thus if $\xi\in \mathfrak{g}$ and $\nu=M_\mathcal{B}(\xi)\in \mathbb{R}^e$ 
is its coordinate vector with respect to
$\mathcal{B}$, then
\begin{equation}
 \label{eqn:pull_back_metric}
\|\xi_M(m)\|^2_m=\|C_m\,\nu\|^2=\nu^t\,C^t\,C\,\nu,
\end{equation}
where $\|\cdot\|_m$ is the norm on $T_mM$, and $\|\cdot\|$ is the standard Euclidean norm on $\mathbb{R}^e$.

Although $C_m$ depends on the choice of the orthonormal basis $\mathcal{B}$ and $\mathcal{B}_m$, $\det (C_m)$ is
invariantly defined (up to sign), and has the following geometric significance. 
Let $G\cdot m\subseteq M$ be the $G$-orbit of $m\in M'$, so
that the choice of an orientation on $\mathfrak{g}$ determines an orientation
on $G\cdot m$; then
$V_{\mathrm{eff}}(m)$ is, by definition, the volume of $G\cdot m$ with respect to the Riemannian volume form for the
restricted metric. We may assume without loss that $\mathcal{B}$ and $\mathcal{B}_m$ have been chosen oriented, so that
$\det (C_m)>0$. 
In view of (\ref{eqn:pull_back_metric}), the pull-back to $G$ of the Riemannian volume form on $G\cdot m$ under the 
$\left|G^M_m\right|:1$ covering map
$\mu^m:G\rightarrow G\cdot m$, $g\mapsto \mu_g(m)$, is $\det(C_m)\,\mathrm{vol}_G$, where $\mathrm{vol}_G$ is 
the Haar volume form; $\det(C_m)$ is clearly $G$-invariant, since $G$ acts on $M$ by Riemannian isometries. 
Therefore,
\begin{equation}
 \label{eqn:orbital_volume}
\left|G^M_m\right|\,V_{\mathrm{eff}}^M(m)=\int_G\det(C_m)\,\mathrm{vol}_G=\det(C_m).
\end{equation}

\section{Proof of Theorem \ref{thm:rapid_decay}.}

\begin{proof}
We shall first prove the Theorem under the assumption $x=y$. 

Let $\mathrm{d}V_G$ be the Haar measure on $G$.
Let $\rho:G\rightarrow U(H)$ be a unitary $G$-action on a Hilbert space $H$; then the orthogonal 
projection $P^{(\varpi)}:H\rightarrow H^{(\varpi)}$ onto the $\varpi$-th isotype, 
with irreducible representation $(\rho_{(\varpi)},V_\varpi)$, is given by
$$
P^{(\varpi)}=d_\varpi\cdot \int_G \chi_{\varpi}\left(g^{-1}\right)\,\rho(g)\,\mathrm{d}V_G(g),
$$
where $d_\varpi=:\dim(V_\varpi)$ \cite{dix}, $\chi_\varpi(g)=:\mathrm{trace}\big(\rho_\varpi(g)\big)$. In our case, 
$U_T^{(\varpi)}(\tau)=P^{(\varpi)}\circ U_T(\tau)=P^{(\varpi)}\circ U(\tau)\circ \Pi=U(\tau)^{(\varpi)}\circ \Pi$,
where $U^{(\varpi)}(\tau)=:P^{(\varpi)}\circ U(\tau)$. Thus, in terms of distributional kernels,
\begin{eqnarray}
 \label{eqn:kernel_composition}
U^{(\varpi)}(\tau)(x,y)&=&d_\varpi\cdot \int_G\overline{\chi_\varpi (g)}\,U(\tau)\left(\mu^X_{g^{-1}}(x),y\right)\,\mathrm{d}V_G(g)
\nonumber\\
&=&d_\varpi\cdot \int_G\overline{\chi_\varpi (g)}\,U(\tau)\left(x,\mu^X_{g}(y)\right)\,\mathrm{d}V_G(g),
\end{eqnarray}
and on the other hand
\begin{equation}
 \label{eqn:kernel_composition_Pi}
U_T^{(\varpi)}(\tau)\left(x',x''\right)=\int_X\,U^{(\varpi)}(\tau)\left(x',y\right)\,\Pi\left(y,x''\right)\,
\mathrm{d}V_X(y).
\end{equation}

Using (\ref{eqn:kernel_composition}) and (\ref{eqn:kernel_composition_Pi}) 
we obtain for (\ref{eqn:traveling_average_wave_operator}) along the diagonal:
\begin{eqnarray}
 \label{eqn:kernel_composition_Pi_average}
\lefteqn{S_{\chi\cdot e^{-i\lambda(\cdot)}}^{(\varpi)}\left(x,x\right)}\\
&=&d_\varpi\cdot \int_G\,\int_X\,\int_{-\epsilon}^{\epsilon}\,\chi(\tau)\,e^{-i\lambda\tau}\,
\overline{\chi_\varpi (g)}\,U(\tau)\left(x,\mu^X_{g}(y)\right)\,
\Pi\left(y,x\right)\nonumber\\
&&\cdot \mathrm{d}\tau\,\mathrm{d}V_X(y)\,\mathrm{d}V_G(g).\nonumber
\end{eqnarray}

Let $X_1\subseteq X$ be an arbitrarily small open neighborhood of $x\in X$, and let $\varrho_x\in \mathcal{C}_0^\infty(X_1)$ 
be identically $=1$ in an open neighborhood $X_2\Subset X_1$ of $x$. Write
$$
S_{\chi\cdot e^{-i\lambda(\cdot)}}^{(\varpi)}\left(x,x\right)=S_{\chi\cdot e^{-i\lambda(\cdot)}}^{(\varpi)}\left(x,x\right)'
+S_{\chi\cdot e^{-i\lambda(\cdot)}}^{(\varpi)}\left(x,x\right)'',
$$
where in the former summand the integrand in (\ref{eqn:kernel_composition_Pi_average}) has been multiplied by $\varrho _z(y)$,
and in the latter by $1-\varrho_x(y)$.

\begin{lem}
\label{eqn:cut_off_1}
As $\lambda\rightarrow\infty$, we have
$$S_{\chi\cdot e^{-i\lambda(\cdot)}}^{(\varpi)}\left(x,x\right)''=O\left(\lambda^{-\infty}\right).$$
\end{lem}

\begin{proof}
On the support of $1-\varrho_x$, we have $\mathrm{dist}_X(y,x)\ge c$ for some fixed $c>0$. Thus, 
$F_x(y)=:\big(1-\varrho_x (y)\big)\,\Pi(y,x)$
is $\mathcal{C}^\infty$ function of $y$. Therefore, the function
$$
\Gamma_{x}(\tau)=: \chi(\tau)\cdot\int_G \overline{\chi_\varpi (g)}\,U(\tau)(F_x)\left(\mu^X_{g^{-1}}(x)\right)\,\mathrm{d}V_G(g)
$$
is $\mathcal{C}^\infty$ and compactly supported, so that its Fourier transform $\widehat{\Gamma}_{x}$ is of rapid decrease.
On the other hand, we have
\begin{eqnarray*}
S_{\chi\cdot e^{-i\lambda(\cdot)}}^{(\varpi)}\left(x,x\right)''=
d_\varpi\cdot \int_{-\epsilon}^{\epsilon}\,e^{-i\lambda\tau}\,\widetilde{F}_{x}(\tau)\,\mathrm{d}\tau=
d_\varpi\cdot\widehat{\Gamma}_{x}(\lambda).
\end{eqnarray*}

\end{proof}

Thus we are reduced to considering the asymptotics of $S_{\chi\cdot e^{-i\lambda(\cdot)}}^{(\varpi)}\left(x,x\right)'$. On the support
of $\Pi'(y,x)=:\varrho_x(y)\,\Pi(y,x)$, we may represent $\Pi$ as an FIO using (\ref{eqn:microlocal_Pi}).

We can adopt the same principle to make one more similar reduction.
In fact, if $\epsilon$ is small enough then the singular support of $U(\tau)$ for $\chi(\tau)\neq 0$
lies within a small tubular neighborhood of the diagonal in $X\times X$. Thus, the contribution to the asymptotics
of $S_{\chi\cdot e^{-i\lambda(\cdot)}}^{(\varpi)}\left(x,x\right)'$ coming from the locus where 
$\mathrm{dist}_X\left(\mu^X_{g}(y),x\right)\ge a>0$ for some suitable small $a>0$ is negligible.  
Since $y$ itself (for $\varrho_x(y)\neq 0$) belongs to a small neighborhood of $x$, in order for
$\mu^X_{g}(y)$ to belong to a small neighborhood of $x$ we need to assume that $g$ belongs to a small neighborhood
of the stabilizer $G^X_m\subseteq G$ of $x$. Therefore, we only lose a negligible contribution to the asymptotics,
if the integrand in (\ref{eqn:kernel_composition_Pi_average}) is further multiplied by a cut-off of the form
$\rho_x(g)$, supported in a small neighborhood of $G^X_m$, and identically $=1$ sufficiently close to
$G^X_m$.

If $G^X_m=\big\{g_1=e,g_2,\ldots,g_r\big\}$ (here $r=r_m$ and $m=\pi(x)$), we may assume for simplicity that
$\rho_x(g)=\sum_{j=1}^r\rho(g_j\,g)$, where $\rho$ is a fixed cut-off supported in a small neighborhood of 
the unit $e\in G$.

Furthermore, on the support of the cut-offs introduced, up to smoothing terms contributing negligibly
to the asymptotics $U(\tau)$ and $\Pi$
may be described as FIOs using (\ref{eqn:microlocal_V_tau}) and (\ref{eqn:microlocal_Pi}). 

Thus we conclude that asymptotically for $\lambda\rightarrow \infty$
\begin{eqnarray}
\label{eqn:S_FIO_1}
\lefteqn{S_{\chi\cdot e^{-i\lambda(\cdot)}}^{(\varpi)}\left(x,x\right)}\\
&\sim& \frac{d_\varpi}{(2\pi)^{2d+1}}\cdot \int_0^{+\infty}\,\int_{\mathbb{R}^{2d+1}}\,\int_G\,\int_X\,
\int_{-\epsilon}^{\epsilon}\,e^{i[\varphi(\tau,x,\eta)-\langle g\cdot y,\eta\rangle
+t\psi(y,x)
-\lambda\tau]}\,
\nonumber\\
&&\cdot \chi(\tau)\,\overline{\chi_\varpi (g)}\,\varrho_x(y)\,\rho_x(g)\,
a(\tau,x,g\cdot y,\eta)\,\,s(y,x,t)\nonumber\\
&&\cdot \mathrm{d}\tau\,\mathrm{d}V_X(y)\,\mathrm{d}V_G(g)\,\mathrm{d}\eta\,\mathrm{d}t  \nonumber\\
&=&\frac{d_\varpi}{(2\pi)^{2d+1}}\cdot \int_0^{+\infty}\,\int_{\mathbb{R}^{2d+1}}\,\int_G\,\int_X\,
\int_{-\epsilon}^{\epsilon}\,e^{i\,\Psi_1}\,\mathcal{A}_1\,\mathrm{d}\tau\,\mathrm{d}V_X(y)\,
\mathrm{d}V_G(g)\,\mathrm{d}\eta\,\mathrm{d}t,\nonumber
\end{eqnarray}
where we have set $g\cdot y=:\mu^X_{g}(y)$. Furthermore in view of 
(\ref{eqn:Hamilton_Jacobi}) the phase $\Psi$ is given by
\begin{eqnarray}
 \label{eqn:phase_1}
 \Psi_1&=:&\varphi(\tau,x,\eta)-\langle g\cdot y,\eta\rangle
+t\psi(y,x)
-\lambda\tau\\
&=& \langle x,\eta\rangle+\tau\,\sigma_Q(x,\eta)-\langle g\cdot y,\eta\rangle
+t\psi(y,x)
-\lambda\tau+\|\eta\|\,O\left(\tau^2\right).    \nonumber
\end{eqnarray}
The amplitude on the other had is given by
\begin{eqnarray}
 \label{eqn:amplitude_1}
 \mathcal{A}_1&=:&\chi(\tau)\,\overline{\chi_\varpi (g)}\,\varrho_x(y)\,\rho_x(g)\,
 a(\tau,x,g\cdot y,\eta)\,\,s(y,x,t).
\end{eqnarray}

For $\lambda\ll 0$, we have
$\partial_\tau\Psi_1\ge C\,(\|\eta\|+|\lambda|)$,
and integrating by parts in $\mathrm{d}\tau$ shows that the right hand side
of (\ref{eqn:S_FIO_1}) is $O\left(\lambda^{-\infty}\right)$ for $\lambda\rightarrow-\infty$.
This proves the first statement of the Proposition.

Let us then focus on the asymptotics for $\lambda\rightarrow+\infty$. To this end,
let us operate the change of variables $t\mapsto \lambda\,t$, $\eta\mapsto \lambda\,\eta$,
so that (\ref{eqn:S_FIO_1}) may be rewritten:

\begin{eqnarray}
\label{eqn:S_FIO_2}
\lefteqn{S_{\chi\cdot e^{-i\lambda(\cdot)}}^{(\varpi)}\left(x,x\right)}\\
&\sim& 2\pi\,d_\varpi\,\left(\frac{\lambda}{2\pi}\right)^{2d+2}\cdot \int_0^{+\infty}\,\int_{\mathbb{R}^{2d+1}}\,\int_G\,\int_X\,
\int_{-\epsilon}^{\epsilon}\,e^{i\,\lambda\Psi_2}\,\mathcal{A}_2\nonumber\\
&&\cdot\mathrm{d}\tau\,\mathrm{d}V_X(y)\,
\mathrm{d}V_G(g)\,\mathrm{d}\eta\,\mathrm{d}t,\nonumber
\end{eqnarray}
where now
\begin{eqnarray}
 \label{eqn:phase_2}
 \Psi_2&=:&\langle x-g\cdot y,\eta\rangle+\tau\,\sigma_Q(x,\eta)
 +t\psi(y,x)-\tau+\|\eta\|\,O\left(\tau^2\right),   
\end{eqnarray}
\begin{eqnarray}
 \label{eqn:amplitude_2}
 \mathcal{A}_2&=:&\chi(\tau)\,\overline{\chi_\varpi (g)}\,\varrho_x(y)\,\rho_x(g)\,
 a\big(\tau,x,g\cdot y,\lambda\,\eta\big)\,\,s\big(y,x,\lambda\,t\big).
\end{eqnarray}

If we set $\eta=r\,\Omega$, with $r>0$ and $\Omega\in S^{2d}\subseteq \mathbb{R}^{2d+1}$
(the unit sphere), we may further rewrite (\ref{eqn:phase_2}) as 
\begin{eqnarray}
\label{eqn:S_FIO_3}
\lefteqn{S_{\chi\cdot e^{-i\lambda(\cdot)}}^{(\varpi)}\left(x,x\right)}\\
&\sim& 2\pi\,d_\varpi\,\left(\frac{\lambda}{2\pi}\right)^{2d+2}\cdot \int_0^{+\infty}\,
\int_0^{+\infty}\,\int_{S^{2d}}\,\int_G\,\int_X\,
\int_{-\epsilon}^{\epsilon}\,e^{i\,\lambda\Psi_3}\,\mathcal{A}_3\nonumber\\
&&\cdot r^{2d}\mathrm{d}\tau\,\mathrm{d}V_X(y)\,
\mathrm{d}V_G(g)\,\mathrm{d}\Omega\,\mathrm{d}r\,\mathrm{d}t,\nonumber
\end{eqnarray}
with
\begin{eqnarray}
 \label{eqn:phase_3}
 \Psi_3&=:& r\langle x-g\cdot y,\Omega\rangle+r\,\tau\,\sigma_Q(x,\Omega)
+t\psi(y,x)
-\tau+r\,O\left(\tau^2\right),   
\end{eqnarray}
\begin{eqnarray}
 \label{eqn:amplitude_3}
 \mathcal{A}_3&=:&\chi(\tau)\,\overline{\chi_\varpi (g)}\,\varrho_x(y)\,\rho_x(g)\,
 a\big(\tau,x,g\cdot y,\lambda\,r\,\Omega\big)\,\,s\big(y,x,\lambda\,t\big).
\end{eqnarray}

So far we have not made a specific choice of local coordinates. It is now convenient to assume
that the computation is being carried out in a system of Heisenberg local coordinates centered at $x$,
and we write $y=x+(\theta,\mathbf{v})$, with the replacement
\begin{equation}
 \label{eqn:replacement_X}
 \int_X\,\mathrm{d}V_X(y)\,\longrightarrow\,\int_{\mathbb{R}^{2d}}\int_{-\pi}^\pi\,
\mathcal{V}_M(\theta,\mathbf{v})\,\mathrm{d}\theta\,\mathrm{d}\mathbf{v},
\end{equation}
where $\mathcal{V}_M$ is the local coordinate expression of the volume density on $M$, 
and in particular
$\mathcal{V}_M(\theta,\mathbf{0})=1/(2\pi)$.

We shall write accordingly 
\begin{equation}
 \eta=\left(\eta',\eta''\right)=r\,\Omega=r\cdot\left(\Omega',\mathbf{\Omega}\right)
 \in \mathbb{R}\times \mathbb{R}^{2d}\,\,\,\mathrm{with}\,\,\,
 \left(\Omega'\right)^2+\left\|\mathbf{\Omega}\right\|^2=1,
\end{equation}
and make the replacement
\begin{equation}
 \label{eqn:replacement_eta}
 \int_{\mathbb{R}^{2d+1}}\mathrm{d}\eta\,\longrightarrow\,
 \int_0^{+\infty}r^{2d}\,\mathrm{d}r\,\int_{S^{2d}}\mathrm{d}\Omega
\end{equation}

In Heisenberg local coordinates, $\eta=(1,\mathbf{0})$ corresponds the cotangent vector
$\alpha_x$. Using this and (\ref{eqn:wave_front_Pi}), (\ref{eqn:wave_front_tau})
one can prove the following:

\begin{lem}
\label{lem:cut_off_omega}
 Only a negligible contribution to the asymptotics is lost in (\ref{eqn:S_FIO_3}),
if the amplitude $\mathcal{A}_3$ is multiplied by a cut off function $\gamma_1(\Omega)$,
compactly supported in a small neighborhood $S_1\subseteq S^{2d}$
of $(1,\mathbf{0})$ and identically $=1$
in a smaller neighborhood of $(1,\mathbf{0})$.
\end{lem}

\begin{proof}
See Lemma 2.2 of \cite{p_weyl}.
\end{proof}

How small the neighborhoods in Lemma \ref{lem:cut_off_omega} may be chosen depends on how small
$\epsilon$ is. 

Now $S_1\subseteq S_+$, the upper hemisphere, and on $S_+$
we have a system of local coordinates for $S^{2d}$, given by
$\mathbf{\Omega}\in B_{2d}(\mathbf{0},1)\mapsto \big(\sqrt{1-\|\mathbf{\Omega}\|^2},\,\mathbf{\Omega}\big)$;
we can then make the replacement
\begin{equation}
 \label{eqn:replaced_S_+}
 \int_{S^{2d}}\,\mathrm{d}\Omega\longrightarrow \int_{B_{2d}(\mathbf{0},1)}\mathcal{V}_{S}(\mathbf{\Omega})\,\mathrm{d}\mathbf{\Omega},
\end{equation}
where $\mathcal{V}_{S}(\mathbf{\Omega})$ is the local coordinate expression of the volume form on the sphere,
and integration is compactly supported.

Integration in
$\mathrm{d}t\,\mathrm{d}r$ may also assumed to be compactly supported:

\begin{lem}
 \label{lem:cut_off_t_r}
 Only a negligible contribution to the asymptotics is lost in (\ref{eqn:S_FIO_3}),
if for some $D\gg 0$ the amplitude $\mathcal{A}_3$ is further multiplied by a cut off function $\gamma_2(t,r)$,
compactly supported in $(1/D,D)^2$ and identically $=1$
in $(2/D,D/2)^2$.
\end{lem}

\begin{proof}
This follows integrating by parts in $(\tau,y)$, by arguments on the line of those
in Lemma 2.3 of \cite{p_weyl}.
\end{proof}

At this point, integrating parts in $\mathrm{d}t$ will show that integration
in $y$ may be restricted to a suitable shrinking neighborhood of $x$; more precisely, we have:

\begin{lem}
\label{lem:shrinking_v_theta}
 For any $C,\,a>0$ the locus where $\|(\theta,\,\mathbf{v})\|\ge C\,\lambda^{a-1/2}$ contributes
 negligibly to the asymptotics of (\ref{eqn:S_FIO_3}).
\end{lem}

\begin{proof}
By Corollary 1.3 of \cite{bdm_sj}, for some $C'>0$ we have at any $(y,z)\in X\times X$
$$
\Im \psi(y,z)\ge C\,\mathrm{dist}(y,z)^2.
$$
Therefore, on the locus where $\|(\theta,\,\mathbf{v})\|\ge C\,\lambda^{a-1/2}$
we have
$$
\big|\partial_t\Psi_3\big|=
\big|\psi\big(x+(\theta,\mathbf{v}),x\big)\big|\ge \Im \psi\big(x+(\theta,\mathbf{v})\big)
\ge C''\,\lambda^{2a-1}.
$$
Iteratively integrating by parts in $\mathrm{d}t$ then introduces at each step a factor
$\lambda^{-2a}$.
\end{proof}

To fix ideas, we shall take in the following $a=1/24$. 
Summing up, multiplying the amplitude in (\ref{eqn:S_FIO_3}) by
\begin{equation}
\label{eqn:defn_gamma}
 \beta_\lambda (\Omega,t,r,\theta,\mathbf{v})=:
\gamma_1(\Omega)\,\gamma_2(t,r)\,\gamma_3\left(\lambda^{11/24}\,\|(\theta,\mathbf{v})\|\right)
\end{equation}
does not alter the asymptotics.

We can now at least verify that (\ref{eqn:S_FIO_3}) is rapidly decreasing if $x$ does not belong to a \textit{fixed}
small tubular neighborhood of $X'$.

\begin{lem} 
Given $\epsilon>0$, there exists $\delta'>0$, which may chosen very small if $\epsilon$ is sufficiently small,
such that
$$
S_{\chi\cdot e^{-i\lambda(\cdot)}}^{(\varpi)}\left(x,x\right)
=O\left(\lambda^{-\infty}\right)
$$ for
$\mathrm{dist}_X(x,X')\ge \delta'$.
\end{lem}

\begin{proof}
We have restricted integration in $\mathrm{d}\mathbf{\Omega}$ to the locus where $\|\mathbf{\Omega}\|<\delta_1$, say,
so that $1\ge \Omega'\ge 1-c\delta_1^2$ for some $c>0$.

Furthermore, since $\mathbf{0}\in \mathfrak{g}^\vee$ is a regular value of $\Phi$, we have for some
$C_1>0$ that (with $m=\pi(x)$)
$$
\big\|\Phi(m)\big\|\ge C_1\,\mathrm{dist}_M\left(\pi(x),M'\right) 
=C_1\,\mathrm{dist}_X\left(x,X'\right).
$$

On the other hand, integration in $\mathrm{d}V_G(g)$ is localized in a small neighborhood of $G^X_m$.
Let $\mathbf{e}_G:\mathfrak{g}\rightarrow G$
denote the exponential map
Thus any $g$ in a neighborhood of $g_\ell\in G^X_m$ may be uniquely written 
$g=\mathbf{e}_G(\xi)\,g_\ell$, for some $\xi\in B_{\mathfrak{g}}(\mathbf{0},\delta)$;
the latter denotes
an open ball of some small radius $\delta>0$ centered at the origin in the Lie algebra of $G$.
Again, $\delta$ may be chosen arbitrarily small at the price of making $\epsilon$ itself
small enough. Thus in 
the range $\|(\theta,\mathbf{v})\|\le C\,\lambda^{-11/24}$ and $\|\mathbf{\Omega}\|\le \delta_1$, we have 
$$
-\langle g\cdot y,\Omega\rangle=r\,\Big[\langle \Phi(m),\xi\rangle\,\Omega'-\langle\xi^\sharp(m),\mathbf{\Omega}
\rangle+O\left(\|\xi\|^2\right)+O\left(\lambda^{-11/24}\right)\Big].
$$
It follows in view of (\ref{eqn:phase_3}) that for $\lambda\gg 0$ and some fixed $b>0$
\begin{eqnarray}
 \label{eqn:gradient_xi}
\big\|\nabla_\xi \Psi_3\big\|&=&r\,\Big\|\Phi(m)\,\Omega'-F_m(\mathbf{\Omega})+O(\|\xi\|)\Big\|\nonumber\\
&\ge&
\frac 1D\,\left[C_1\,\left(1-\delta^2\right)\,\mathrm{dist}_X\left(x,X'\right)-b\,(\delta_1+\delta)\right]
\end{eqnarray}
(here $F_m$ is an appropriate linear map).

We see from (\ref{eqn:gradient_xi}) that $\big\|\nabla_\xi \Psi_3\big\|$ is bounded below by a fixed positive constant
for 
\begin{equation}
 \label{eqn:localized_X'}
\mathrm{dist}_X\left(x,X'\right)\ge \frac{2}{C_1}\,b\,(\delta_1+\delta).
\end{equation}

Here $C_1$ and $b$ are fixed, while $\delta_1,\,\delta$ may be taken arbitrarily small with $\epsilon$ small enough.
Integrating by parts in $\mathrm{d}V_G(g)$, we conclude that the locus (\ref{eqn:localized_X'}) contributes negligibly to the
asymptotics of (\ref{eqn:S_FIO_3}), as claimed.

\end{proof}

Therefore, we shall assume in the following that $x$ belongs to some small tubular neighborhood of $X'$ in $X$.

Let us also fix an orthonormal basis $(\xi_j)$ of $\mathfrak{g}$. 
On the other hand, the choice of $(\xi_j)$ determines a unitary isomorphism $\mathfrak{g}\cong
\mathbb{R}^e$ (the latter with the standard Euclidean structure). We shall thus identify 
$\xi =\sum_j\nu_j\,\xi_j\in \mathfrak{g}$ with its coordinate vector $\nu\in \mathbb{R}^e$,
and  with this understanding write $g=\mathbf{e}_G(\nu)\,g_a$. 
Similarly, we shall write $\nu_M\in \mathfrak{X}(M)$ and $\nu_X\in \mathfrak{X}(X)$ 
for the vector fields generated on $M$ and $X$ by $\nu\in \mathbb{R}^e\cong \mathfrak{g}$.

Let $\Phi_j=\langle \Phi,\xi_j\rangle$ be the components of the moment map with respect to
the $\xi_j$'s; then under the previous identification 
\begin{equation}
 \label{eqn:identification_moment}
 \langle \Phi,\xi\rangle=\langle \Phi,\nu\rangle=\sum_j\Phi_j\,\nu_j.
\end{equation}

Since the metric on $G$ is bi-invariant,
we can make the replacement
\begin{equation}
 \label{eqn:replacement_G}
 \int_G\,\rho_x(g)\,\mathrm{dV}_G(g)\longrightarrow\,\sum_\ell \int_{B_e(\mathbf{0},\delta)}\,\mathcal{V}_G(\nu)\,
\rho\big(\mathbf{e}_G(\nu)\big)\,\mathrm{d}\nu,
\end{equation}
where now $B_e(\mathbf{0},\delta)\subseteq \mathbb{R}^e$ is the open ball centered at the origin and radius $\delta$,
and $\rho$ is a fixed cut off supported in it; $\mathcal{V}_G(\nu)$
is the local coordinate expression of the Haar density. Dependence on $\ell$ is of course in the rest of the amplitude
(see below).

Before proceeding, let us show that a cut-off similar to the one in Lemma
\ref{lem:shrinking_v_theta}
may be applied to integration in $\mathrm{d}\nu$, in each summand of (\ref{eqn:replacement_G}) (we take $a=1/9$ for concreteness):

\begin{lem}
\label{lem:shrinking_xi}
If $C$ is as in Lemma \ref{lem:shrinking_v_theta} and $C_1\gg C$ then 
the locus where $\|\nu\|\ge C_1\,\lambda^{-11/24}$ contributes
 negligibly to the asymptotics of (\ref{eqn:S_FIO_3}).
\end{lem}

\begin{proof}
 Let us go back to (\ref{eqn:phase_2}), which may be rewritten as
\begin{eqnarray}
 \label{eqn:phase_2_decomposed}
 \Psi_2&=:&\big[t\psi(y,x)-\tau\big]+\langle x-g\cdot y,\eta\rangle+\tau\,F(x,\eta,\tau),   
\end{eqnarray}
where
$F(x,\eta,\tau)=:\sigma_Q(x,\eta)
 +\|\eta\|\,O\left(\tau\right)$. 

In a natural manner, $T_x^*X\hookrightarrow T_{(x,\beta)}\left(T^*X\right)$ for any $(x,\beta)\in T^*X$.
In particular, we can view $\alpha_x$ as a tangent vector to $\Sigma\subseteq T^*X$ at $(x,r\alpha_x)$ for some $r>0$,
and interpret $\nabla_\eta\Psi_2$ as a tangent vector to $x$ by duality.

If $y=x+(\theta,\mathbf{v})$ with $\|(\theta,\mathbf{v})\|\le C\,\lambda^{-11/24}$
then the Euclidean gradient of $\Psi_2$ with respect to $\eta$ is
\begin{equation}
\label{eqn:gradient_eta}
\nabla_\eta\Psi_2=-\xi_X(x)+O\left(\lambda^{-11/24}\right)+\tau\,\nabla_\eta\,F.
\end{equation}

Now if $x\in X'$ then 
$\langle \alpha_x,\xi_X(x)\rangle=0$; on the other hand
$$F(x,r\,\alpha_x,\tau)=r\,\varsigma_T(x,\eta)
 +r\,O\left(\tau\right),$$ so that $\langle \alpha_x,\nabla_\eta\,F\rangle=d_\eta F(\alpha_x)=\varsigma_T(x,\eta)
 +O\left(\tau\right)\ge C_2>0$ for some $C_2>0$, if $\epsilon$ has been chosen small enough.

It follows that for any $x\in X'$ we have
$$\mathfrak{g}_X(x)\oplus \mathrm{span}\big(\nabla_\eta\,F\big)=\{\mathbf{0}\}$$ over $X'$, and by continuity the same holds 
in an open (fixed) tubular neighborhood of $X_1\subseteq X$ of $X'$; here $\mathfrak{g}_X(x)\subseteq T_xX$ is the image 
of the evaluation map $\mathrm{ev}_x:\xi\in \mathfrak{g}\mapsto \xi_X(x)\in T_xX$. Since $\mathrm{ev}_x$ is injective
near $X'$, we conclude from (\ref{eqn:gradient_eta}) that over $X_1$ for some $D_1>0$ we have:
\begin{equation}
 \big\|\nabla_\eta\Psi_2\big\| \ge D_1\,\|\xi\|+O\left(\lambda^{-11/24}\right)
= D_1\,\|\nu\|+O\left(\lambda^{-11/24}\right).
\end{equation}
Thus, if $C_1\gg 0$ and $\|\nu\|\ge C_1\,\lambda^{-11/24}$ then
$\big\|\nabla_\eta\Psi_2\big\| \ge D_2\,\lambda^{-11/24}$ for some $D_2>0$.
The claim then follows as in Lemma \ref{lem:shrinking_v_theta}
integrating by parts in $\mathrm{d}\eta$, which is legitimate since integration in $\mathrm{d}\eta$
is now compactly supported.

\end{proof}

Arguing as for Lemma 3.2 of \cite{pao_JMP} (with $\mu^X_{g_\ell}$ in place of
$\phi^X_{-\tau_0}$)
we have in Heisenberg local coordinates
\begin{eqnarray}
 \label{eqn:action_local_coordinates_1}
g_\ell\cdot y&=&
g_\ell\cdot \big(x+(\theta,\mathbf{v})\big)\\
&=&g_\ell\cdot\big(x+(\theta,\mathbf{v})\big)=
x+\left(\theta+R_3(\mathbf{v}),A_\ell\mathbf{v}+R_2(\mathbf{v})\right).\nonumber
\end{eqnarray}

Here and in the following $R_j$ will denote a generic function, allowed to vary from line to line, 
defined on some open neighborhood
of the origin in a Euclidean space, and vanishing to $j$-th order at $\mathbf{0}$.

Given this, in view of
Corollary 2.2 of \cite{pao_IJM} applied with $\vartheta=-\nu$ we obtain
\begin{eqnarray}
 \label{eqn:action_local_coordinates_2}
\mathbf{e}_G(\nu)\,g_\ell\cdot y&=&\mathbf{e}_G(\nu)\cdot 
\Big(x+\big(\theta+R_3(\mathbf{v}),A_\ell\mathbf{v}+R_2(\mathbf{v})\big)\Big)\\
&=&x+\Big(\Theta_\ell(\theta,\mathbf{v},\nu),
\mathbf{V}_\ell(\mathbf{v},\nu)\Big),\nonumber
\end{eqnarray}
where
$$
\Theta_\ell(\theta,\mathbf{v},\nu)=:\theta-\langle\Phi(m),\nu\rangle-\omega_m\big(\nu_M(m),A_\ell\mathbf{v})+R_3(\nu,\mathbf{v}),
$$
$$
\mathbf{V}_\ell (\mathbf{v},\nu)=:A_\ell \mathbf{v}+\nu_M(m)+R_2(\nu,\mathbf{v}).
$$
Thus, 
\begin{equation}
 \label{eqn:decomposition_stabilizer}
 S_{\chi\cdot e^{-i\lambda(\cdot)}}^{(\varpi)}\left(x,x\right)\sim \sum_{\ell=1}^{r_m}
 S_{\chi\cdot e^{-i\lambda(\cdot)}}^{(\varpi)}\left(x,x\right)^{(\ell)},
\end{equation}
where the sum is over $G^X_m$, and the $\ell$-th summand is given by
\begin{eqnarray}
\label{eqn:S_FIO_4}
\lefteqn{S_{\chi\cdot e^{-i\lambda(\cdot)}}^{(\varpi)}\left(x,x\right)^{(\ell)}}\\
&\sim& 2\pi\,d_\varpi\,\left(\frac{\lambda}{2\pi}\right)^{2d+2}\cdot \int_{1/D}^{D}\,
\int_{1/D}^{D}\,\int_{B_{2d}(\mathbf{0},1)}\,\int_{\mathbb{R}^e}\,\int_{\mathbb{R}^{2d}}\,\int_{-\pi}^\pi
\int_{-\epsilon}^{\epsilon}\,e^{i\,\lambda\Psi_4^{(\ell)}}\,\mathcal{A}_4^{(\ell)}\nonumber\\
&&\cdot \mathcal{V}_{S}(\mathbf{\Omega})\,\mathcal{V}_G(\nu)\,
\mathcal{V}(\theta,\mathbf{v})\,r^{2d}\mathrm{d}\tau\,\mathrm{d}\theta\,\mathrm{d}\mathbf{v}\,
\mathrm{d}\nu\,\mathrm{d}\mathbf{\Omega}\,\mathrm{d}r\,\mathrm{d}t,\nonumber
\end{eqnarray}
where now, in view of (\ref{eqn:expansion_sz}),
\begin{eqnarray}
 \label{eqn:phase_4}
 \Psi_4^{(\ell)}&=:& -r\big\langle (\Theta_\ell,\,\mathbf{V}_\ell),\left(\Omega',\,\mathbf{\Omega}\right)\big\rangle+r\,\tau\,\sigma_Q(x,\Omega)
-\tau  \\
&&+t\psi\Big(x+(\theta,\mathbf{v}),x\Big)+r\,O\left(\tau^2\right)\nonumber\\
&=&-r\,\Theta_\ell\,\Omega'-r\,\langle \mathbf{V}_\ell,\mathbf{\Omega}\rangle+r\,\tau\,\sigma_Q(x,\Omega)
-\tau  +r\,O\left(\tau^2\right)\nonumber\\
&&+it\,\left[1-e^{i\theta}\right]+\frac{i}{2}\,t\,\|\mathbf{v}\|^2
\,e^{i\theta}+t\,R_3\left(\mathbf{v}\right)\,e^{i\theta}\nonumber\\
&=&-r\,\Big[\theta-\langle\Phi(m),\nu\rangle-\omega_m\big(\nu_M(m),A_\ell\mathbf{v})+R_3(\nu,\mathbf{v})\Big]\,\Omega'\nonumber\\
&&-r\,\Big\langle A_\ell\mathbf{v}+\nu_M(m)+R_2(\nu,\mathbf{v}),\mathbf{\Omega}\Big\rangle+r\,\tau\,\sigma_Q(x,\Omega)
-\tau  \nonumber\\
&&+it\,\left[1-e^{i\theta}\right]+\frac{i}{2}\,t\,\|\mathbf{v}\|^2
\,e^{i\theta}+t\,R_3\left(\mathbf{v}\right)\,e^{i\theta}+r\,O\left(\tau^2\right)\nonumber\\
&=&it\,\left[1-e^{i\theta}\right]-r\,\theta\,\Omega'+r\,\tau\,\sigma_Q(x,\Omega)\nonumber
-\tau+r\,O\left(\tau^2\right)\nonumber\\
&&+r\,\Big[\langle\Phi(m),\nu\rangle\,\Omega'-\langle A_\ell\mathbf{v}+\nu_M(m),\mathbf{\Omega}\rangle\Big]
\nonumber\\
&&+r\,\omega_m\big(\nu_M(m),A_\ell\mathbf{v})\,\Omega'+\frac{i}{2}\,t\,\|\mathbf{v}\|^2\,e^{i\theta}
-r\,\big\langle R_2(\nu,\mathbf{v}),\mathbf{\Omega}\big\rangle\nonumber\\
&&-r\,R_3(\nu,\mathbf{v})\,\Omega'+t\,R_3\left(\mathbf{v}\right)\,e^{i\theta}.\nonumber
\end{eqnarray}
\begin{eqnarray}
 \label{eqn:amplitude_4}
 \mathcal{A}_4^{(\ell)}&=:&\chi(\tau)\,\overline{\chi_\varpi \big(\mathbf{e}_G(\xi)\,g_\ell\big)}\,
 \varrho_x\big(x+(\theta,\mathbf{v})\big)\,\rho\big(\mathbf{e}_G(\nu)\big)\,
 \beta_\lambda (\Omega,t,r,\theta,\mathbf{v},\nu)
 \nonumber\\
&&\cdot  a\Big(\tau,x,g\cdot \big(x+(\theta,\mathbf{v})\big),\lambda\,r\,\Omega\Big)\,
s\big(x+(\theta,\mathbf{v}),x,\lambda\,t\big),
\end{eqnarray}
where $\Omega=\big(\Omega',\mathbf{\Omega}\big)=\big(\sqrt{1-\|\mathbf{\Omega}\|^2},\,\mathbf{\Omega}\big)$,
$g=\mathbf{e}_G(\xi)\,g_\ell$. We have set
\begin{equation}
 \label{eqn:defn_beta_lambda}
\beta_\lambda (\Omega,t,r,\theta,\mathbf{v},\nu)=:
\beta_\lambda (\Omega,t,r,\theta,\mathbf{v})\cdot \gamma_5\left(\lambda^{11/24}\,\nu\right),
\end{equation}
for an appropriate cut-off $\gamma_5$ supported near the origin in $\mathfrak{g}^\vee$.

Here the expression $\Phi_\nu(m)=\langle\Phi(m),\nu\rangle$ refers to the pairing 
$\mathfrak{g}^\vee\times \mathfrak{g}\rightarrow \mathbb{R}$, while
$\langle A_\ell\mathbf{v}+\nu_M(m),\mathbf{\Omega}\rangle$ refers to the pairing
$T_mM\times T_mM^\vee\rightarrow \mathbb{R}$.

Now let us introduce the rescaled variables
\begin{equation}
 \label{eqn:1st_rescaling}
 \mathbf{v}\mapsto \frac{1}{r\sqrt{\lambda}}\,\mathbf{v},\,\,\,\nu\mapsto\frac{1}{r\sqrt{\lambda}}\,\nu.
\end{equation}
Thus integration in $\mathrm{d}\mathbf{v}\,\mathrm{d}\nu$ will now be on an expanding ball in $\mathbb{R}^{2d}\times \mathbb{R}^e$ 
of radius $O\left(\lambda^{1/24}\right)$.
Then (\ref{eqn:S_FIO_4}) may be further rewritten
\begin{eqnarray}
 \label{eqn:S_FIO_5}
\lefteqn{S_{\chi\cdot e^{-i\lambda(\cdot)}}^{(\varpi)}\left(x,x\right)^{(\ell)}\,\sim\,
\frac{d_\varpi}{(2\pi)^{2d+1}}\,\lambda^{d+2-e/2}}\\
&& \cdot \int_{1/D}^{D}\,
\int_{1/D}^{D}\,\int_{B_{2d}(\mathbf{0},1)}\,\int_{\mathbb{R}^e}\,\int_{\mathbb{R}^{2d}}\,\int_{-\pi}^\pi
\int_{-\epsilon}^{\epsilon}\,e^{i\lambda \,\Gamma+i\,\sqrt{\lambda}\,\Upsilon}
\,e^{B^{(\ell)}}\,e^{\lambda\,R_3\left(\frac{\nu}{\sqrt{\lambda}},\frac{\mathbf{v}}{\sqrt{\lambda}}\right)}\nonumber\\
&&\cdot\mathcal{A}^{(\ell)}\cdot \mathcal{V}_{S}(\mathbf{\Omega})\,\mathcal{V}_G\left(\frac{\nu}{r\sqrt{\lambda}}\right)\,
\mathcal{V}\left(\theta,\frac{\mathbf{v}}{r\sqrt{\lambda}}\right)
\,r^{-e}\mathrm{d}\tau\,\mathrm{d}\theta\,\mathrm{d}\mathbf{v}\,
\mathrm{d}\nu\,\mathrm{d}\mathbf{\Omega}\,\mathrm{d}r\,\mathrm{d}t,\nonumber
\end{eqnarray}
where, in view of (\ref{eqn:phase_4}) and (\ref{eqn:amplitude_4})
we have
\begin{eqnarray}
\label{eqn:phase_5}
\Gamma&=:&it\,\left[1-e^{i\theta}\right]-r\,\theta\,\Omega'+r\,\tau\,\sigma_Q(x,\Omega)
-\tau+r\,O\left(\tau^2\right)\\
\Upsilon^{(\ell)}&=:&\langle\Phi(m),\nu\rangle\,\Omega'-\langle A_\ell\mathbf{v}+\nu_M(m),\mathbf{\Omega}\rangle
\nonumber\\
B^{(\ell)}&=:&\frac ir\,\omega_m\big(\nu_M(m),A_\ell\mathbf{v})\,\Omega'-\frac{1}{2\,r^2}\,t\,\|\mathbf{v}\|^2\,e^{i\theta}
-\frac ir\,\big\langle R_2(\nu,\mathbf{v}),\mathbf{\Omega}\big\rangle,\nonumber
\end{eqnarray}
while $\mathcal{A}^{(\ell)}$ is $\mathcal{A}_4^{(\ell)}$ expressed in terms of the new rescaled variables.

We may then rewrite (\ref{eqn:S_FIO_5}) the following form:
\begin{eqnarray}
 \label{eqn:S_FIO_6}
\lefteqn{S_{\chi\cdot e^{-i\lambda(\cdot)}}^{(\varpi)}\left(x,x\right)^{(\ell)}
\,\sim\, (2\pi)^{-(2d+1)}\, d_\varpi\,\lambda^{d+2-e/2}}\\
&&\cdot\int_{\mathbb{R}^e}\,\left[\int_{B_{2d}(\mathbf{0},1)}\,\int_{\mathbb{R}^{2d}}
\,e^{i\,\sqrt{\lambda}\,\Upsilon^{(\ell)}}
\,I_\lambda^{(\ell)}(\mathbf{v},\mathbf{\Omega},\nu)\,\mathrm{d}\mathbf{v}\,\mathrm{d}\mathbf{\Omega}\right]
\,\mathrm{d}\nu,\nonumber
\end{eqnarray}
where
\begin{equation}
 \label{eqn:I_lambda_v_Omega}
 I_\lambda(\mathbf{v},\mathbf{\Omega},\nu)^{(\ell)}=:\int_{1/D}^{D}\,
\int_{1/D}^{D}\,\int_{-\epsilon}^{\epsilon}\,\int_{-\pi}^\pi\,e^{i\lambda \,\Gamma}\,\mathcal{B}^{(\ell)}\,
\mathrm{d}\theta\,\mathrm{d}\tau\,\mathrm{d}r\,\mathrm{d}t,
\end{equation}
with
\begin{equation}
 \label{eqn:amplitude_B}
 \mathcal{B}^{(\ell)}=:e^{B^{(\ell)}}\cdot e^{\lambda\,R_3\left(\frac{\nu}{\sqrt{\lambda}},\frac{\mathbf{v}}{\sqrt{\lambda}}\right)}
\,\mathcal{A}^{(\ell)}
\cdot \mathcal{V}_{S}(\mathbf{\Omega})\,\mathcal{V}_G\left(\frac{\nu}{r\sqrt{\lambda}}\right)\,
\mathcal{V}\left(\theta,\frac{\mathbf{v}}{r\sqrt{\lambda}}\right)
\,r^{-e}.
\end{equation}

For $\|\nu\|,\,\|\mathbf{v}\|\le C\,\lambda^{1/24}$, we have
$$
\lambda\,R_3\left(\frac{\nu}{\sqrt{\lambda}},\frac{\mathbf{v}}{\sqrt{\lambda}}\right)
=O\left(\lambda^{-3/8}\right).
$$
Therefore, in the same range
$$
|\mathcal{B}^{(\ell)}|\le C'\,\lambda^d\,e^{-c\,\|\mathbf{v}\|^2}
$$
for some $C',\,c>0$. Also, we have an asymptotic expansion, coming from the asymptotic expansions
of $a$ and $s$ as classical symbols and from the Taylor expansions in the rescaled variables,
$$
\mathcal{A}^{(\ell)}
\cdot \mathcal{V}_{S}(\mathbf{\Omega})\,\mathcal{V}_G\left(\frac{\nu}{r\sqrt{\lambda}}\right)\,
\mathcal{V}\left(\theta,\frac{\mathbf{v}}{r\sqrt{\lambda}}\right)
\,r^{-e}\sim \sum_{k\ge 0}\lambda^{d-k/2}\,P_k^{(\ell)}(\nu,\mathbf{v}),
$$
where $P_k^{(\ell)}(\nu,\mathbf{v})$ is a polynomial of joint degree $\le k$ (with coefficients depending
on all the other variables, which we omit).
It follows that $\mathcal{B}^{(\ell)}$ has an asymptotic expansion of the form:
\begin{equation}
 \label{eqn:amplitude_B_expanded}
 \mathcal{B}^{(\ell)}\sim e^{B^{(\ell)}}\cdot \sum_{k\ge 0}\lambda^{d-k/2}\,Q_k^{(\ell)}(\nu,\mathbf{v},\mathbf{\Omega}),
\end{equation}
where $Q_k^{(\ell)}(\nu,\mathbf{v},\mathbf{\Omega})$ is a polynomial in
$(\nu,\mathbf{v})$, of joint degree $\le 3\,k$ (again, with coefficients depending
on all the other variables).
Furthermore, the leading coefficient is
$$
Q_0^{(\ell)}=\left(\dfrac{t}{\pi}\right)^{d}\,r^{-e}\cdot \overline{\chi_\varpi(g_\ell)}.
$$

Let us evaluate $I_\lambda^{(\ell)}(\mathbf{v},\mathbf{\Omega},\nu)$ asymptotically for $\lambda\rightarrow+\infty$,
viewing it as an oscillatory integral in $(\theta,t,\tau,r)$, with oscillatory parameter $\lambda$ and phase
$\Gamma$ with non-negative imaginary part. 
A computation that we leave to the reader shows that for $|\tau|<\epsilon$ and $\epsilon$ small
enough we have:

\begin{lem}
$\Gamma$ has a unique stationary point
$$
P_0=(\theta_0,t_0,\tau_0,r_0)=:\left(0,\dfrac{\Omega'}{\sigma_Q(x,\Omega)},0,\dfrac{1}{\sigma_Q(x,\Omega)}\right).
$$
The Hessian matrix at the stationary point satisfies
$$
\det \left(\dfrac{\lambda}{2\pi\,i}\,\mathrm{H}(\Psi)(P_0)\right)=\left(\dfrac{\lambda}{2\pi}\right)^4\,\sigma_Q(x,\Omega)^2.
$$
In particular, the stationary point is non-degenerate.
\end{lem}

We have $\Gamma(P_0)=0$ and
\begin{eqnarray*}
 B^{(\ell)}(P_0,\mathbf{v},\mathbf{\Omega},\nu)&=&\sigma_Q(x,\Omega)\,\Omega'\,\left[i\,\omega_m\big(\nu_M(m),A_\ell\mathbf{v})
-\frac{1}{2}\,\|\mathbf{v}\|^2\right]\\
&&
- i\,\big\langle R_2(\nu,\mathbf{v}),\mathbf{\Omega}\big\rangle
\end{eqnarray*}

Applying the stationary phase lemma, we obtain an asymptotic expansion
\begin{eqnarray}
\label{eqn:I_lambda_expanded}
 I_\lambda(\mathbf{v},\mathbf{\Omega},\nu)^{(\ell)}&\sim&\dfrac{(2\pi)^2}{\pi^d}\,e^{B^{(\ell)}(P_0,\mathbf{v},\mathbf{\Omega},\nu)}\,\lambda^{d-2}\,\sigma_Q(x,\Omega)^{e-(1+d)}
\cdot{\Omega'}^d\cdot
\,\overline{\chi_\varpi(g_\ell)}\nonumber\\
&&\cdot \left[1+\sum_{j\ge 1}\lambda^{-j/2}\widetilde{Q}_j^{(l)}(\nu,\mathbf{v})\right].
\end{eqnarray}
 
Since the remaining integration is compactly supported in $\mathbf{\Omega}$, and 
over an expanding ball of radius $O\left(\lambda^{1/24}\right)$, the expansion may be integrated term by term. 

Let us next view the integral in $\mathrm{d}\mathbf{v}\,\mathrm{d}\mathbf{\Omega}$ in (\ref{eqn:S_FIO_6})
as an oscillatory integral with parameter $\sqrt{\lambda}$ and real phase $\Upsilon^{(\ell)}$.
Again, a computation that we leave to the reader and application of Lemma \ref{lem:signature_1}
with $r=2d$ and $R=-A$ (a symplectic matrix) shows the following:

\begin{lem}
 \label{lem:critical_point_Upsilon}
$\Upsilon^{(\ell)}$ has a unique critical point, given by
$$
(\mathbf{v}_0,\mathbf{\Omega}_0)=:\left(-A_\ell^{-1}\nu_M(m),\mathbf{0}\right).
$$
The Hessian matrix at the critical point has the form
$$
\mathrm{Hess}\left(\Upsilon^{(\ell)}\right)_{(\mathbf{v}_0,\mathbf{\Omega}_0)}=\begin{pmatrix}
                                                       0_{2d}&-A_{\ell}^t\\
                                                       -A_{\ell}&-\Phi_\nu(m)\,I_{2d}
                                                      \end{pmatrix},
$$
where $\Phi_\nu=\langle \Phi,\nu\rangle$. In particular, 
its determinant and signature are 
$$
\det\left(\mathrm{Hess}\left(\Upsilon^{(\ell)}\right)_{(\mathbf{v}_0,\mathbf{\Omega}_0)}\right)=1,\,\,\,\,\,\,
\mathrm{sgn}\left(\mathrm{Hess}\left(\Upsilon^{(\ell)}\right)_{(\mathbf{v}_0,\mathbf{\Omega}_0)}\right)=0.
$$
\end{lem}

Also, we have (recalling that $A_\ell$ is unitary (i.e., symplectic and orthogonal): 
\begin{eqnarray}
 \label{eqn:exponent_critical}
\lefteqn{i\sqrt{\lambda}\,\Upsilon^{(\ell)}(P_0,\mathbf{v}_0,\mathbf{\Omega}_0,\nu)+B^{(\ell)}(P_0,\mathbf{v}_0,\mathbf{\Omega}_0,\nu)}\\
&=&i\sqrt{\lambda}\cdot\langle\Phi(m),\nu\rangle-\frac{1}{2}\,\varsigma_T(x)\cdot\|\nu_M(m)\|^2.
\end{eqnarray}

Thus we obtain for the inner integral in (\ref{eqn:S_FIO_6}) an asymptotic expansion
\begin{eqnarray}
 \label{eqn:inner_integral_expanded}
\lefteqn{\int_{B_{2d}(\mathbf{0},1)}\,\int_{\mathbb{R}^{2d}}
\,e^{i\,\sqrt{\lambda}\,\Upsilon^{(\ell)}}
\,I_\lambda^{(\ell)}(\mathbf{v},\mathbf{\Omega},\nu)\,\mathrm{d}\mathbf{v}\,\mathrm{d}\mathbf{\Omega}}\\
&\sim& \dfrac{1}{\pi^d}\,(2\pi)^{2d+2}\,\lambda^{-2}\,\varsigma_T(x)^{e-(1+d)}\,\overline{\chi_\varpi(g_\ell)} \cdot
\nonumber\\
&&\cdot e^{i\sqrt{\lambda}\cdot\langle\Phi(m),\nu\rangle-\frac{1}{2}\,\varsigma_T(x)\cdot\|\nu_M(m)\|^2}
\sum_{j\ge 0}\lambda^{-j/2}\,R_j(\nu)
\end{eqnarray}
where $R_j$ is a polynomial in $\nu$ (of degree $\le 3j$), and $R_0=1$.

The expansion may be integrated in $\mathrm{d}\nu$. 
To perform the computation, let $C_m$ be the matrix representing the evaluation map 
$\mathrm{val}_m:\mathfrak{g}\rightarrow \mathfrak{g}_M(m)$, $\xi\mapsto \xi_M(m)$,
with respect to the orthonormal basis $(\xi_j)$ of $\mathfrak{g}$ and an orthonormal
basis of $\xi_M(m)\subseteq T_mM$. Thus
$\|\nu_M(m)\|^2=\nu^t\,C_m^tC_m\,\nu$. If we set $\mathbf{s}=C_m\,\nu$, and then 
$\mathbf{r}=:\sqrt{\varsigma_T(x)}\cdot \mathbf{s}$, we get
for every $j$
\begin{eqnarray}
 \label{eqn:j_th_integral}
\lefteqn{\int_{\mathbb{R}^e}R_j(\nu)\,e^{i\sqrt{\lambda}\cdot\langle\Phi(m),\nu\rangle-\frac{1}{2}\,\varsigma_T(x)\cdot\|\nu_M(m)\|^2}\,
\mathrm{d}\nu}\nonumber\\
&=&\frac{1}{\det(C_m)}\,\int_{\mathbb{R}^e}\widetilde{R}_j(\mathbf{s})\,
e^{i\sqrt{\lambda}\cdot\big\langle (C^{-1})^t\Phi(m),\mathbf{s}\big\rangle
-\frac{1}{2}\,\varsigma_T(x)\cdot\|\mathbf{s}\|^2}\,
\mathrm{d}\mathbf{s}\nonumber\\
&=&\frac{1}{\det(C_m)}\,\varsigma_T(x)^{-e/2}\int_{\mathbb{R}^e}\widehat{R}_j(\mathbf{r})\,
e^{i\sqrt{\lambda}\cdot\big\langle L\big(\Phi(m)\big),\mathbf{r}\big\rangle
-\frac{1}{2}\cdot\|\mathbf{r}\|^2}\,
\mathrm{d}\mathbf{r}
\end{eqnarray}
where $L\big(\Phi(m)\big)=:\varsigma_T(x)^{-1/2}\,(C^{-1})^t\Phi(m)$, while $\widetilde{R}_j$ and $\widehat{R}_j$
are obtained from $R_j$ by substitution.

For $j=0$, we get
\begin{eqnarray}
 \label{eqn:0_th_term}
\lefteqn{\frac{1}{\det(C_m)}\,\varsigma_T(x)^{-e/2}\int_{\mathbb{R}^e}\,
e^{i\sqrt{\lambda}\cdot\big\langle L\big(\Phi(m)\big),\mathbf{s}\big\rangle
-\frac{1}{2}\cdot\|\mathbf{r}\|^2}\,
\mathrm{d}\mathbf{r}}\\
&=&\frac{1}{\det(C_m)}\,(2\pi)^{e/2}\,\varsigma_T(x)^{-e/2}\,
\exp\left(-\frac 12\,\lambda\,\big\|L\big(\Phi(m)\big)\big\|^2\right).\nonumber
\end{eqnarray}

For general $j$, on the other hand we obtain
\begin{eqnarray}
 \label{eqn:j_th_term}
\lefteqn{\int_{\mathbb{R}^e}\widehat{R}_j(\mathbf{r})\,
e^{i\sqrt{\lambda}\cdot\big\langle L\big(\Phi(m)\big),\mathbf{r}\big\rangle
-\frac{1}{2}\cdot\|\mathbf{r}\|^2}\,
\mathrm{d}\mathbf{r}}\\
&=&\frac{1}{\det(C_m)}\,\varsigma_T(x)^{-e/2}\,S_j\left(\sqrt{\lambda}\,L\big(\Phi(m)\big)\right)\,
\exp\left(-\frac 12\,\lambda\,\big\|L\big(\Phi(m)\big)\big\|^2\right),\nonumber
\end{eqnarray}
where $S_j$ is a polynomial of degree $\le 3j$, and $\big\|L\big(\Phi(m)\big)\big\|^2\ge c\,\|\Phi(m)\|^2$ for some
$c>0$. Thus for every $N\gg 0$ we can write
\begin{eqnarray}
\label{eqn:ell_th_term_expanded}
S_{\chi\cdot e^{-i\lambda(\cdot)}}^{(\varpi)}\left(x,x\right)^{(\ell)}
&\sim&\sum_{j=0}^N\,\lambda^{-j/2}\,Q_j^{(\ell)}\left(\sqrt{\lambda}\,\Phi(m)\right)\,
\exp\left(-\frac 12\,\lambda\,\big\|L\big(\Phi(m)\big)\big\|^2\right)\nonumber\\
&&+
O\left(\lambda^{-(N+1)}\right),
\end{eqnarray}
where $Q_j^{(\ell)}$ is a polynomial of degree $\le 3j$.

If as assumed $\mathrm{dist}_X(x,X')=\mathrm{dist}_M(m,M')\ge C\,\lambda^{-\frac{11}{24}}$, 
then $\|\Phi(m)\|\ge C'\,\lambda^{-\frac{11}{24}}$
for some $C'>0$, because $\mathbf{0}\in \mathfrak{g}^\vee$ is a regular value of $\Phi$; so that 
\begin{equation}
 \label{eqn:exponentian_decay}
\exp\left(-\frac 12\,\lambda\,\big\|L\big(\Phi(m)\big)\big\|^2\right)\le 
\exp\left(-\frac 12\,C'\,\lambda^{1/12}\right).
\end{equation}
In the case $x=y$, Theorem \ref{thm:rapid_decay} follows from (\ref{eqn:ell_th_term_expanded}) and (\ref{eqn:exponentian_decay}).

Given this, the general case now follows from the Cauchy-Schwartz inequality:
\begin{eqnarray*}
\lefteqn{\left| S_{\chi\cdot e^{-i\lambda(\cdot)}}^{(\varpi)}(x,y)\right|}\nonumber\\
&=&
\left|\sum_j\widehat{\chi}\big(\lambda-\lambda_j^{(\varpi)}\big)^{1/2}\,\,e_j^{(\varpi)}(x)\cdot 
\overline{\widehat{\chi}\big(\lambda-\lambda_j^{(\varpi)}\big)^{1/2}\,e_j^{(\varpi)}(y)}\right|\nonumber\\
&\le&\sqrt{\sum_j\widehat{\chi}\big(\lambda-\lambda_j^{(\varpi)}\big)\,\,\left|e_j^{(\varpi)}(x)\right|^2}
\cdot \sqrt{\sum_j\widehat{\chi}\big(\lambda-\lambda_j^{(\varpi)}\big)\,\,\left|e_j^{(\varpi)}(y)\right|^2}\nonumber\\
&=&S_{\chi\cdot e^{-i\lambda(\cdot)}}^{(\varpi)}(x,x)^{1/2}\,S_{\chi\cdot e^{-i\lambda(\cdot)}}^{(\varpi)}(y,y)^{1/2}.
\end{eqnarray*}

\end{proof}

\section{Proof of Theorem \ref{thm:local_scaling_asymtpotics}}

\begin{proof}
It suffices to prove the Theorem in case $\theta_2=0$. Let us set for $\lambda>0$
$$
x_{1\lambda}=:x+\left(\frac{\theta_1}{\sqrt{\lambda}},\frac{\mathbf{w}_1}{\sqrt{\lambda}}\right),\,\,\,\,
x_{2\lambda}=:x+\frac{\mathbf{w}_2}{\sqrt{\lambda}}.
$$
Following the same line of argument as in the proof of Theorem \ref{thm:rapid_decay},
we obtain the analogues of (\ref{eqn:decomposition_stabilizer}) and (\ref{eqn:S_FIO_4}):
\begin{equation}
 \label{eqn:decomposition_stabilizer_lambda}
 S_{\chi\cdot e^{-i\lambda(\cdot)}}^{(\varpi)}\left(x_{1\lambda},x_{2\lambda}\right)\sim \sum_{\ell=1}^{r_m}
 S_{\chi\cdot e^{-i\lambda(\cdot)}}^{(\varpi)}\left(x_{1\lambda},x_{2\lambda}\right)^{(\ell)},
\end{equation}
and
\begin{eqnarray}
\label{eqn:S_FIO_4_lambda}
\lefteqn{S_{\chi\cdot e^{-i\lambda(\cdot)}}^{(\varpi)}\left(x_{1\lambda},x_{2\lambda}\right)^{(\ell)}}\\
&\sim& 2\pi\,d_\varpi\,\left(\frac{\lambda}{2\pi}\right)^{2d+2}\cdot \int_{1/D}^{D}\,
\int_{1/D}^{D}\,\int_{B_{2d}(\mathbf{0},1)}\,\int_{\mathbb{R}^e}\,\int_{\mathbb{R}^{2d}}\,\int_{-\pi}^\pi
\int_{-\epsilon}^{\epsilon}\,e^{i\,\lambda\Psi_{4\lambda}^{(\ell)}}\,\mathcal{A}_{4\lambda}^{(\ell)}\nonumber\\
&&\cdot \mathcal{V}_{S}(\mathbf{\Omega})\,\mathcal{V}_G(\nu)\,
\mathcal{V}(\theta,\mathbf{v})\,r^{2d}\mathrm{d}\tau\,\mathrm{d}\theta\,\mathrm{d}\mathbf{v}\,
\mathrm{d}\nu\,\mathrm{d}\mathbf{\Omega}\,\mathrm{d}r\,\mathrm{d}t,\nonumber
\end{eqnarray}
where now
\begin{eqnarray}
 \label{eqn:phase_4_lambda}
 \Psi_{4\lambda}^{(\ell)}&=:& 
it\,\left[1-e^{i\theta}\right]-r\,\theta\,\Omega'+r\,\tau\,\sigma_Q(x_{1\lambda},\Omega)
-\tau+r\,O\left(\tau^2\right)\\
&&+r\,\left[\left(\frac{\theta_1}{\sqrt{\lambda}}+\langle\Phi(m),\nu\rangle\right)\,\Omega'
+\left\langle \frac{\mathbf{w}_1}{\sqrt{\lambda}}-\big( A_\ell\mathbf{v}+\nu_M(m)\big),\mathbf{\Omega}\right\rangle\right]
\nonumber\\
&&+r\,\omega_m\big(\nu_M(m),A_\ell\mathbf{v})\,\Omega'-i\,t\,\psi_2\left(\mathbf{v},\frac{\mathbf{w}_2}{\sqrt{\lambda}}\right)\,e^{i\theta}\nonumber\\
&&-r\,\left\langle R_2\left(\nu,\mathbf{v},\frac{\mathbf{w}_j}{\sqrt{\lambda}}\right),\mathbf{\Omega}\right\rangle
+R_3\left(\nu,\mathbf{v},\frac{\mathbf{w}_j}{\sqrt{\lambda}}\right),\nonumber
\end{eqnarray}
\begin{eqnarray}
 \label{eqn:amplitude_4_lambda}
 \mathcal{A}_{4\lambda}^{(\ell)}&=:&\chi(\tau)\,\overline{\chi_\varpi \big(\mathbf{e}_G(\xi)\,g_\ell\big)}\,
 \varrho_x\big(x+(\theta,\mathbf{v})\big)\,\rho\big(\mathbf{e}_G(\nu)\big)\,
 \beta_\lambda (\Omega,t,r,\theta,\mathbf{v},\nu)
 \nonumber\\
&&\cdot  a\Big(\tau,x_{1\lambda},g\cdot \big(x+(\theta,\mathbf{v})\big),\lambda\,r\,\Omega\Big)\,
s\big(x+(\theta,\mathbf{v}),x_{2\lambda},\lambda\,t\big).
\end{eqnarray}
Here $R_j(a,b,\ldots)$ denotes as before a function of $a,b,\ldots$ vanishing to $j$-th order at the
origin $a=0,\,b=0,\cdots$, and possibly depending on other variables, which are omitted.

%

 %

Given that we have now reduced integration to a shrinking 
domain where $(\mathbf{v},\theta,\nu)=O\left(\lambda^{-11/24}\right)$, and by assumption
$\|(\theta_1,\mathbf{w}_1)\|/\sqrt{\lambda}<C\,\lambda^{-11/24}$, we see from
(\ref{eqn:phase_4_lambda}) that
\begin{equation}
 \label{eqn:r-derivative}
 \partial_r\Psi_{4\lambda}^{(\ell)}=\tau\,
 \left[\sigma_Q(x_{1\lambda},\Omega)+O\left(\tau\right)\right]+O\left(\lambda^{-11/24}\right).
\end{equation}

Since $\sigma_Q(x_{1\lambda},\Omega)$ is bounded from below by a positive constant,
(\ref{eqn:r-derivative}) implies $\left|\partial_r\Psi_{4\lambda}^{(\ell)}\right|\ge C_1\,|\tau|+O\left(\lambda^{-11/24}\right)$;
thus, where $|\tau|\ge 2\,D'\,\lambda^{-11/24}$ for some $D'\gg 0$ we have
$\left|\partial_r\Psi_{4\lambda}^{(\ell)}\right|\ge D'\,\lambda^{-11/24}$.
Again, integration by parts in $\mathrm{d}r$ implies that the corresponding contribution to the asymptotics
of (\ref{eqn:S_FIO_4_lambda}) is negligible. Thus we can multiply the amplitude (\ref{eqn:amplitude_4_lambda}) by a further cut-off of the form
$\gamma_6\left(\lambda^{11/24}\,\tau\right)$, which we implicitly incorporate in $\beta_\lambda$, without altering the asymptotics.

Let us now operate the change of variables
\begin{equation}
 \label{eqn:rescaled_variables}
 \mathbf{v}\mapsto \frac{\mathbf{v}}{\sqrt{\lambda}},\,\,\,\nu\mapsto \frac{\nu}{\sqrt{\lambda}},
 \,\,\,\theta\mapsto \frac{\theta}{\sqrt{\lambda}},\,\,\,\tau\mapsto \frac{\tau}{\sqrt{\lambda}},
\end{equation}
and rewrite (\ref{eqn:S_FIO_4_lambda}) in the form
\begin{eqnarray}
\label{eqn:S_FIO_4_lambda_rescaled}
\lefteqn{S_{\chi\cdot e^{-i\lambda(\cdot)}}^{(\varpi)}\left(x_{1\lambda},x_{2\lambda}\right)^{(\ell)}
\sim 2\pi\,d_\varpi\,\left(\frac{\lambda}{2\pi}\right)^{2d+2}\,\lambda^{-d-1-e/2}}\\
&&\cdot \int_{1/D}^{D}\,
\int_{1/D}^{D}\,\int_{B_{2d}(\mathbf{0},1)}\,\int_{\mathbb{R}^e}\,\int_{\mathbb{R}^{2d}}\,\int_{-\infty}^\infty
\int_{-\infty}^{\infty}\,e^{i\,\lambda\widetilde{\Psi}_{4\lambda}^{(\ell)}}\,
\widetilde{\mathcal{A}}_{4\lambda}^{(\ell)}\nonumber\\
&&\cdot \mathcal{V}_{S}(\mathbf{\Omega})\,\mathcal{V}_G\left(\frac{\nu}{\sqrt{\lambda}}\right)\,
\mathcal{V}\left(\frac{\theta}{\sqrt{\lambda}},\frac{\mathbf{v}}{\sqrt{\lambda}}\right)\,r^{2d}\mathrm{d}\tau\,\mathrm{d}\theta\,\mathrm{d}\mathbf{v}\,
\mathrm{d}\nu\,\mathrm{d}\mathbf{\Omega}\,\mathrm{d}r\,\mathrm{d}t,\nonumber
\end{eqnarray}
where $ \widetilde{\Psi}_{4\lambda}^{(\ell)}$ and $\widetilde{\mathcal{A}}_{4\lambda}^{(\ell)}$ are
$ \Psi_{4\lambda}^{(\ell)}$ and $\mathcal{A}_{4\lambda}^{(\ell)}$, respectively, with the rescaled
variables inserted. Explicitly, keeping in mind that $\Phi(m)=\mathbf{0}$ because $m\in M'$,
we have
\begin{eqnarray}
 \label{eqn:tilde_Psi_lambda}
 \lefteqn{ \widetilde{\Psi}_{4\lambda}^{(\ell)}}\\
 &=&
  \frac{1}{\sqrt{\lambda}}\cdot\Big\{t\,\theta-\tau
  +r\,\Big[\tau\,\sigma_Q(x,\Omega)+(\theta_1-\theta)\,\Omega'
  +\langle \mathbf{w}_1
  -\big(A_\ell\mathbf{v}+\nu_M(m)\big),\mathbf{\Omega}\rangle\Big]\Big\}\nonumber\\
  &&+\frac{1}{\lambda}\,\Big[r\tau\,\big\langle \partial_x\sigma_Q(x,\Omega),\mathbf{w}_1\big\rangle
  +r\,O\left(\tau^2\right)+r\,\omega_m\big(\nu_M(m),A_\ell\mathbf{v})\,\Omega'\nonumber\\
&&  -i\,t\,\psi_2(\mathbf{v},\mathbf{w}_2)\,e^{i\theta/\sqrt{\lambda}}
-r\,\big\langle R_2(\nu,\mathbf{v},\mathbf{w}_j),\mathbf{\Omega}\big\rangle\Big] 
+R_3\left(\frac{\mathbf{\tau}}{\sqrt{\lambda}},\frac{\mathbf{v}}{\sqrt{\lambda}},
\frac{\mathbf{w}_j}{\sqrt{\lambda}}, \frac{\nu}{\sqrt{\lambda}}\right)\nonumber\\
&=&\frac{1}{\sqrt{\lambda}}\,K^{(\ell)}_\nu+\frac{1}{\lambda}\,H^{(\ell)}_\nu+
R_3\left(\frac{\theta}{\sqrt{\lambda}},\frac{\mathbf{\tau}}{\sqrt{\lambda}},\frac{\mathbf{v}}{\sqrt{\lambda}},
\frac{\mathbf{w}_j}{\sqrt{\lambda}}, \frac{\nu}{\sqrt{\lambda}}\right),\nonumber
\end{eqnarray}
where
\begin{eqnarray}
 \label{eqn:phase_nu}
\lefteqn{K^{(\ell)}_\nu(t,\theta,r,\tau,\mathbf{v},\mathbf{\Omega})}\\
&=:&t\,\theta-\tau 
+r\,\Big[\tau\,\sigma_Q(x,\Omega)+(\theta_1-\theta)\,\Omega'+\langle \mathbf{w}_1
  -\big(A_\ell\mathbf{v}+\nu_M(m)\big),\mathbf{\Omega}\rangle\Big],\nonumber
\end{eqnarray}
\begin{eqnarray}
 \label{eqn:amplitude_nu}
\lefteqn{ H^{(\ell)}_\nu(t,\theta,r,\tau,\mathbf{v},\mathbf{\Omega})=:r\tau\,\big\langle \partial_x\sigma_Q(x,\Omega),\mathbf{w}_1\big\rangle
  +r\,O\left(\tau^2\right)}\\
  &&+r\,\omega_m\big(\nu_M(m),A_\ell\mathbf{v})\,\Omega' 
   -i\,t\,\psi_2(\mathbf{v},\mathbf{w}_2)
-r\,\big\langle R_2(\nu,\mathbf{v},\mathbf{w}_j),\mathbf{\Omega}\big\rangle.\nonumber
\end{eqnarray}
Here $O\left(\tau^2\right)$ is to be interpreted as a homogeneous quadratic term, since terms in $\tau$ of order $\ge 3$
have been incorporated in $R_3$, and similarly for $R_2(\nu,\mathbf{v},\mathbf{w}_j)$. In particular, 
$H^{(\ell)}_\nu(t,\theta,r,\tau,\mathbf{v},\mathbf{\Omega})$ is homogeneous of degree two in the rescaled variables.

Thus
\begin{eqnarray}
\label{eqn:S_FIO_4_lambda_rescaled_bis}
S_{\chi\cdot e^{-i\lambda(\cdot)}}^{(\varpi)}\left(x_{1\lambda},x_{2\lambda}\right)^{(\ell)}
=(2\pi)^{-2d-1}\,d_\varpi\,\lambda^{d+1-e/2}\cdot \int_{\mathbb{R}^e}\,I(\nu,\lambda)\,\mathrm{d}\nu,
\end{eqnarray}
where
\begin{eqnarray}
\label{eqn:inner_integral_expanded_lambda}
 I(\nu,\lambda)&=&\int_{1/D}^{D}\,
\int_{1/D}^{D}\,\int_{B_{2d}(\mathbf{0},1)}\,\int_{\mathbb{R}^{2d}}\,\int_{-\infty}^\infty
\int_{-\infty}^{\infty}\,e^{i\,\sqrt{\lambda}\,K^{(\ell)}_\nu}\cdot e^{iH^{(\ell)}_\nu+\lambda\,R_3}\\
&&\cdot\widetilde{\mathcal{A}}_{4\lambda}^{(\ell)}\cdot \mathcal{V}_{S}(\mathbf{\Omega})\,
\mathcal{V}_G\left(\frac{\nu}{\sqrt{\lambda}}\right)\,
\mathcal{V}\left(\frac{\theta}{\sqrt{\lambda}},\frac{\mathbf{v}}{\sqrt{\lambda}}\right)\,
r^{2d}\,\mathrm{d}\tau\,\mathrm{d}\theta\,\mathrm{d}\mathbf{v}\,
\mathrm{d}\mathbf{\Omega}\,\mathrm{d}r\,\mathrm{d}t\nonumber
\end{eqnarray}

Integration in the rescaled variables is over an expanding ball of radius $O\left(\lambda^{1/24}\right)$.
Let us view (\ref{eqn:inner_integral_expanded_lambda}) as an oscillatory integral in $\sqrt{\lambda}$, with real phase $K^{(\ell)}_\nu$
given by (\ref{eqn:phase_nu}), and amplitude
\begin{eqnarray}
 \label{eqn:amplitude_nu_ell_expanded}
\mathcal{B}^{(\ell)}_\nu&=:&e^{iH^{(\ell)}_\nu+\lambda\,R_3}\cdot \widetilde{\mathcal{A}}_{4\lambda}^{(\ell)}\cdot \mathcal{V}_{S}(\mathbf{\Omega})\,
\mathcal{V}_G\left(\frac{\nu}{\sqrt{\lambda}}\right)\,
\mathcal{V}\left(\frac{\theta}{\sqrt{\lambda}},\frac{\mathbf{v}}{\sqrt{\lambda}}\right)\,
r^{2d}.
\end{eqnarray}

In HLC centered at $x\in X$, the amplitude $s(t,x,y)\sim \sum_{l\ge 0}t^{d-l}\,s_l(x,y)$ 
of $\Pi$ in (\ref{eqn:microlocal_Pi}) satisfies $s_0(x,x)=\pi^{-d}$. The amplitude $a$ 
of $V(\tau)$ in (\ref{eqn:microlocal_V_tau}), on the other hand, 
satisfies $a(0,x,x,\eta)=1/\mathcal{V}(x)=2\pi$ (see the discussion following (7) in \cite{p_weyl}).
Therefore, recalling the rescaling $t\mapsto \lambda\,t$ and $\eta\mapsto \lambda\,\eta$, we obtain for
$\mathcal{B}^{(\ell)}_\nu$
an asymptotic expansion
\begin{eqnarray}
 \label{eqn:amplitude_B_expanded_nu}
\mathcal{B}^{(\ell)}_\nu&\sim& e^{iH^{(\ell)}_\nu}\,r^{2d}\cdot \overline{\chi_\varpi(g_\ell)}\cdot 
\beta_\lambda(\nu,\theta,\tau,\mathbf{v},\mathbf{w}_j)\\
&&\cdot\left(\dfrac{\lambda}{\pi}\right)^d\,\sum_{k\ge 0}\lambda^{-k/2}\,P_k(\nu,\theta,\tau,\mathbf{v},\mathbf{w}_j),\nonumber
\end{eqnarray}
where $P_k\in \mathcal{C}^\infty(t,r,\mathbf{\Omega})[\nu,\theta,\tau,\mathbf{v},\mathbf{w}_j]$ 
is a polynomial in the rescaled variables of joint degree $\le 3k$, $P_0=t^d$,
and $\beta_\lambda$ is a compactly supported
bump function, whose support in the rescaled variables is a ball of radius $O\left(\lambda^{1/24}\right)$. 
The expansion may be integrated term by term, so that 
\begin{equation}
 \label{eqn:I_nu_expanded}
I(\nu,\lambda)\sim \sum _{k\ge 0}I(\nu,\lambda)_k,
\end{equation}
where
\begin{eqnarray}
\label{eqn:I_vu_k}
 I(\nu,\lambda)_k&\sim& \dfrac{\lambda^{d-k/2}}{\pi^d}\,\overline{\chi_\varpi(g_\ell)}\cdot\int_{1/D}^{D}\,
\int_{1/D}^{D}\,\int_{B_{2d}(\mathbf{0},1)}\,\int_{\mathbb{R}^{2d}}\,\int_{-\infty}^\infty
\int_{-\infty}^{\infty}\,e^{i\,\sqrt{\lambda}\,K^{(\ell)}_\nu}\nonumber\\
&&\cdot e^{iH^{(\ell)}_\nu}\,r^{2d}\cdot  
\beta_\lambda(\nu,\theta,\tau,\mathbf{v},\mathbf{w}_j)\,P_k(\nu,\theta,\tau,\mathbf{v},\mathbf{w}_j)\nonumber\\
&&\cdot \mathrm{d}\tau\,\mathrm{d}\theta\,\mathrm{d}\mathbf{v}\,
\mathrm{d}\mathbf{\Omega}\,\mathrm{d}r\,\mathrm{d}t.
\end{eqnarray}

Let us note the following:

\begin{lem}
\label{lem:parity}
 For any $k\ge 0$, $P_k$ has the same parity as $k$.
\end{lem}

\begin{proof}
The asymptotic expansions in $t$ of the amplitude $s$ of $\Pi$ in (\ref{eqn:microlocal_Pi}) and 
in $\eta$ of the amplitude $a$ of $V(\tau)$ in (\ref{eqn:microlocal_V_tau}) go down by integer steps
of degree of homogeneity. Now consider $\mathcal{A}_{4\lambda}^{(\ell)}(x',x'')$ 
given by $\mathcal{A}_{4\lambda}^{(\ell)}$ in (\ref{eqn:amplitude_4_lambda}) 
with $x_{1\lambda}$ and $x_{2\lambda}$
replaced by some fixed $x',x''\in X$. In view of substitutions $t\mapsto \lambda\,t$ 
and $\eta\mapsto \lambda\,\eta$, $\mathcal{A}_{4\lambda}^{(\ell)}$ is given by an asymptotic expansion
in descending integer powers of $\lambda$. Therefore the appearance of fractional powers 
of $\lambda$ in (\ref{eqn:amplitude_B_expanded_nu}) is due solely to Taylor expansion 
in the rescaled variables, and this implies the claim.
\end{proof}

The proof of the following is left to the reader:

\begin{lem}
The following holds:
\begin{enumerate}
 \item $K^{(\ell)}_\nu$ has a unique stationary point
\begin{eqnarray*}
 P_\ell(\nu)&=&(t_0,\theta_0,r_0,\tau_0,\mathbf{v}_0^{(\ell)}(\nu),\mathbf{\Omega}_0)\\
 &=:&
\left(\frac{1}{\varsigma_T(x)},0,\frac{1}{\varsigma_T(x)},0,A_\ell^{-1}\big(\mathbf{w}_1-\nu_M(m)\big),\mathbf{0}\right).
\end{eqnarray*}

\item Let us define the column vector
$$
\mathbf{D}=:\left.\partial_{\mathbf{\Omega}}\ln \sigma_Q\right|_{(x,\alpha_x)}.
$$
Then the Hessian matrix at the critical point is
$$
\mathrm{Hess}(K^{(\ell)}_\nu)_{P_\ell(\nu)}=\begin{pmatrix}
            0&1&0&0&\mathbf{0}^t&\mathbf{0}^t\\
            1&0&-1&0&\mathbf{0}^t&\mathbf{0}^t\\
            0&-1&0&\varsigma_T(x)&\mathbf{0}^t&\mathbf{0}^t\\
            0&0&\varsigma_T(x)&0&\mathbf{0}^t&\mathbf{D}^t\\
            \mathbf{0}&\mathbf{0}&\mathbf{0}&\mathbf{0}&[0]&-A_\ell^t/\varsigma_T(x)\\
            \mathbf{0}&\mathbf{0}&\mathbf{0}&\mathbf{D}&-A_\ell/\varsigma_T(x)&-\theta_1\,I_{2d}
           \end{pmatrix}.
$$
\item The determinant of $\mathrm{Hess}(K^{(\ell)}_\nu)_{P_\ell (\nu)}$
is
$$
\det\left(\mathrm{Hess}(K^{(\ell)}_\nu)_{P_0^{(\ell)}(\nu)}\right)=\varsigma_T(x)^{2-4d}.
$$
\item The signature of $\mathrm{Hess}(K^{(\ell)}_\nu)_{P_0^{(\ell)}(\nu)}$
is zero.
\item At the critical point, we have $K^{(\ell)}_\nu\big(P_0^{(\ell)}(\nu)\big)=\theta_1/\varsigma_T(x)$ and
\begin{eqnarray}
 \label{eqn:exponent_critical_1}
 \lefteqn{i\,H^{(\ell)}_\nu\big(P_0^{(\ell)}(\nu)\big)}\\
 &=&\frac{1}{\varsigma_T(x)}\,
\Big[i\,\omega_m\big(\nu_M(m),\mathbf{w}_1)+\psi_2\big(\mathbf{w}_1-\nu_M(m),\,A_\ell\mathbf{w}_2\big)\Big].
\nonumber
\end{eqnarray}
\item The inverse of the Hessian matrix at the critical point is:
\begin{equation*}
\mathrm{Hess}(K^{(\ell)}_\nu)_{P_\ell(\nu)}^{-1}= \begin{pmatrix}
 0&1&0&1/\varsigma_T(x)&\mathbf{D}^t\,A_\ell&\mathbf{0}^t\\
            1&0&0&0&\mathbf{0}^t&\mathbf{0}^t\\
0&0&0&1/\varsigma_T(x)&\mathbf{D}^t\,A_\ell&\mathbf{0}^t\\
1/\varsigma_T(x)&0&1/\varsigma_T(x)&0&\mathbf{0}^t&\mathbf{0}^t\\
A_\ell^{t}\,\mathbf{D}&\mathbf{0}&A_\ell^{t}\,\mathbf{D}&\mathbf{0}&\theta_1\,\varsigma_T(x)^2\,I_{2d}&-\varsigma_T(x)\,A_\ell^t\\
            \mathbf{0}&\mathbf{0}&\mathbf{0}&\mathbf{0}&-\varsigma_T(x)\,A_\ell&[0]\end{pmatrix}
\end{equation*}
(recall that $A_\ell\,A_\ell^t=I_{2d}$).
\item The third order remainder of the phase at the critical point is
$$
R_K^{(3)}=r\,\Big[\tau\cdot\sigma_Q(x,\Omega)^{(1)}-\theta\cdot\big(\Omega'\big)^{(1)}
  \Big],
$$
where $\sigma_Q(x,\Omega)^{(1)}$ and $\big(\Omega'\big)^{(1)}$ are the first order remainders
at the origin $\mathbf{0}\in \mathbb{R}^{2d}$ as functions of $\mathbf{\Omega}$.
In particular, as far as the rescaled variables are concerned, it only depends
on $\tau$ and $\theta$, and is linear in them.

\end{enumerate}

\end{lem}

Therefore, the gradient of $K^{(\ell)}_\nu$ is bounded below in norm, uniformly in $\nu$ and $\ell$, 
by a fixed positive constant when
$P=(t,\theta,r,\tau,\mathbf{v},\mathbf{\Omega})\in \mathbb{R}^4\times \mathbb{R}^{2d}\times B_{2d}(\mathbf{0},1)$ 
remains at distance $\ge a$ from $P_0$, where $a>0$ is fixed. Let 
$\kappa \in \mathcal{C}^\infty_0\left(\mathbb{R}^4\times \mathbb{R}^{2d}\times B_{2d}(\mathbf{0},1)\right)$ be 
identically $\equiv 1$ at distance $\le a$ from the origin, and
set $\kappa_\nu(P)=:\kappa\big(P-P_0^{(\ell)}(\nu)\big)$. Then we can write $I(\nu,\lambda)=I(\nu,\lambda)'
+I(\nu,\lambda)''$, where $I(\nu,\lambda)'$ and $I(\nu,\lambda)''$ are given by (\ref{eqn:inner_integral_expanded_lambda})
with the amplitude multiplied by $\kappa_\nu(P)$ and $1-\kappa_\nu(P)$, respectively.

\begin{lem}
 \label{lem:I_nu_rapid_decrease}
Uniformly in $\nu$, we have $I(\nu,\lambda)''=O\left(\lambda^{-\infty}\right)$ as $\lambda\rightarrow+\infty$.
\end{lem}

\begin{proof}
Where $1-\kappa_\nu\neq 0$, we can \lq integrate by parts\rq\, in 
$\mathrm{d}t\,\mathrm{d}\theta\,\mathrm{d}r\,\mathrm{d}\tau
\,\mathrm{d}\mathbf{v}\, \mathrm{d}\mathbf{\Omega}$ using the previous remark, and noting that the integrand
is compactly supported (with an expanding support). At each step, as in the usual proof of the stationary phase Lemma,
we get a factor $\lambda^{-1}$; furthermore, differentiation of the amplitude, and of the coefficients 
of the first order operator involved, introduces in view of the factor
$e^{iH^{(\ell)}_\nu}$ and (\ref{eqn:amplitude_nu}) a factor
$O\left(\lambda^{1/24}\lambda^{1/24}\right)=O\left(\lambda^{1/12}\right)$. 
Thus we obtain at each step a factor $O\left(\lambda^{-11/12}\right)$, whence after $N$
step a factor $O\left(\lambda^{-11N/12}\right)$. Since integration
is over a ball of radius $O\left(\lambda^{1/9}\right)$, the statement follows.
\end{proof}

Hence as far at the asymptotics are concerned we need only consider $I(\nu,\lambda)'$, which can be estimated
applying the stationary phase Lemma. Let us first note the following, whose proof is also left to the reader:

\begin{lem}
 \label{lem:exponent_rewritten}
Let the components $\mathbf{w}_{j\mathrm{h}},\,\mathbf{w}_{j\mathrm{v}},\,\mathbf{w}_{j\mathrm{t}}$
be as in (\ref{eqn:hvt}). Then (\ref{eqn:exponent_critical_1}) may be rewritten
\begin{eqnarray*}
 i\,H^{(\ell)}_\nu\big(P_0^{(\ell)}(\nu)\big)&=&\mathcal{T}_\ell(\mathbf{w}_1,\mathbf{w}_2)
\nonumber\\&&+\frac{1}{\varsigma_T(x)}\,\Big[
i\,\omega_0\big(\nu_M(m),\mathbf{w}_{1\mathrm{t}}+A_\ell\mathbf{w}_{2\mathrm{t}}\big)
-\frac 12\,\Big\|\nu_M(m)\Big\|^2\Big],
\end{eqnarray*}
where
\begin{eqnarray*}
 \mathcal{T}_\ell(\mathbf{w}_1,\mathbf{w}_2)&=:&\frac{1}{\varsigma_T(x)}\,
 \Big[\psi_2\big(\mathbf{w}_{1\mathrm{h}},A_\ell\mathbf{w}_{2\mathrm{h}}\big)
-\frac 12\,\big\|\mathbf{w}_{1\mathrm{t}}-A_\ell\mathbf{w}_{2\mathrm{t}}\big\|^2\Big]\\
&&+\frac{i}{\varsigma_T(x)}\,\Big[\omega_0\big(\mathbf{w}_{1\mathrm{v}},\mathbf{w}_{1\mathrm{t}}\big)
-\omega_0\big(\mathbf{w}_{2\mathrm{v}},\mathbf{w}_{2\mathrm{t}}\big)\Big].
\end{eqnarray*}
\end{lem}

\begin{rem}
 Here $\omega_0$ and $\|\cdot\|$ are the standard symplectic structure and norm, respectively,
on $\mathbb{R}^{2d}$, and they correspond to the symplectic structure $\omega_m$ 
and norm $\|\cdot \|_m$ on $T_mM$ in the given HLC
system centered at $x$.
\end{rem}

Let us set
$$
\mathcal{D}=:\big(\partial_t,\,\partial_\theta,\,\partial_r,\,\partial_\tau,\,\partial_\mathbf{v},\partial_\mathbf{\Omega}\big)^T
$$
and
\begin{eqnarray}
\label{eqn:defn_L_K}
L_K&=:&\left\langle \mathcal{D},\mathrm{Hess}(K^{(\ell)}_\nu)_{P_\ell(\nu)}^{-1}\,\mathcal{D}\right\rangle\\
&=&2\,\left[\dfrac{\partial^2}{\partial t\partial\theta}+\frac{1}{\varsigma_T(x)}\,\dfrac{\partial^2}{\partial t\partial\tau}
+\frac{1}{\varsigma_T(x)}\,\dfrac{\partial^2}{\partial r\partial\tau}+
\mathbf{D}^t\,A_\ell\,\left(\dfrac{\partial^2}{\partial t\partial\mathbf{v}}+\dfrac{\partial^2}{\partial r\partial\mathbf{v}}\right)\right.
\nonumber\\
&&\left.-\varsigma_T(x)\,\left\langle \frac{\partial}{\partial\mathbf{\Omega}},A_\ell\frac{\partial}{\partial\mathbf{v}}\right\rangle\right]
+\theta_1\,\varsigma_T(x)^2\,\left\langle \frac{\partial}{\partial\mathbf{v}},\frac{\partial}{\partial\mathbf{v}}\right\rangle.\nonumber
\end{eqnarray}

By the stationary phase Lemma \cite{hor}, each $I(\nu,\lambda)_k$ is given by an asymptotic expansion
of the following form
\begin{eqnarray}
\label{eqn:I_vu_k_expanded}
 I(\nu,\lambda)_k&\sim& (2\pi)^{2+2d}\,\pi^{-d}\,\lambda^{-1-k/2}\,\varsigma_T(x)^{2d-1}\,\overline{\chi_\varpi(g_\ell)}\cdot 
e^{i\sqrt{\lambda}\,\theta_1/\varsigma_T(x)}\nonumber\\
&&\cdot\sum_{j\ge 0}\,\lambda^{-j/2}\,\frac{1}{i^j}\,\left.L_j\left(e^{iH^{(\ell)}_\nu}\, r^{2d}\cdot  
P_k\right)\right|_{P_0}
\end{eqnarray}
where
\begin{equation}
 \label{eqn:L_varphi}
L_j(\varphi)=:\sum_{a-b=j}\,\sum_{2a\ge 3b}\,\dfrac{1}{2^a\,a!\,b!}\,L_K^a\left(\varphi\cdot \left(R_K^{(3)}\right)^b\right);
\end{equation}
notice that $\beta_\lambda$ is identically equal to $1$ on the present restricted domain of integration.
Furthermore $a\le 3j$ in the range of summation.

The leading term in (\ref{eqn:I_vu_k_expanded}) is therefore
\begin{eqnarray*}
\frac{(2\pi)^{2(1+d)}}{\pi^{d}}\,\lambda^{-1-k/2}\,\varsigma_T(x)^{-1-d}\,\overline{\chi_\varpi(g_\ell)}\cdot 
e^{i\sqrt{\lambda}\,\theta_1/\varsigma_T(x)+i\,H^{(\ell)}_\nu\big(P_0^{(\ell)}(\nu)\big)}\cdot  
R_k(P_0),
\end{eqnarray*}
where $R_k(P_0)$ is a polynomial of degree $\le 3k$ in the rescaled variables.

Since furthermore $L_K$ is homogeneous of degree $-1$ in the rescaled variables, while $H^{(\ell)}_\nu$ is homogeneous of degree
$2$ and $R_K^{(3)}$ is linear in them, applying (\ref{eqn:L_varphi}) with $\varphi=e^{iH^{(\ell)}_\nu}\, r^{2d}\cdot  
P_k$, we conclude that the general
term in (\ref{eqn:I_vu_k_expanded}) is a linear combination of terms of the form
$$
\lambda^{-1-(k+j)/2}\, e^{i\sqrt{\lambda}\,\theta_1/\varsigma_T(x)+i\,H^{(\ell)}_\nu\big(P_0^{(\ell)}(\nu)\big)}\,
P_{a,b}(\theta_1,\mathbf{w}_1,\mathbf{w}_2,\nu_M(m)),
$$
where $P_{a,b}$ 
has parity $-a+b+k\equiv a-b+k=j+k$.  

In view of the last summand in (\ref{eqn:defn_L_K}) and of (\ref{eqn:amplitude_nu}), 
where $a-b=j$, $2a\ge 3b$, and $P_{a,b}(P_0)$ is now a polynomial in $(\nu,\theta_1,\mathbf{w}_1,\mathbf{w}_2)$ 
of joint degree
$\le 3k+3a+b\le 11\,(k+j)$. Thus
\begin{eqnarray*}
\lefteqn{\left|\lambda^{-1-(k+j)/2}\, e^{i\sqrt{\lambda}\,\theta_1/\varsigma_T(x)+i\,H^{(\ell)}_\nu\big(P_0^{(\ell)}(\nu)\big)}\,
P_{a,b}(\theta_1,\mathbf{w}_1,\mathbf{w}_2,\nu_M(m))\right|}\\
&\le& C_{k,j}\,\lambda^{-1-\frac{1}{2}\,(k+j)}\, \lambda^{\frac{11}{24}\,(k+j)}
=C_{k,j}\,\lambda^{-1-\frac{1}{24}\,(k+j)}.
\end{eqnarray*}

Since integration in $\mathrm{d}\nu$ is on a domain of radius $O\left(\lambda^{1/24}\right)$
the expansion may be integrated term by term.
Therefore, we obtain for (\ref{eqn:S_FIO_4_lambda_rescaled_bis}) an asymptotic expansion 
\begin{eqnarray}
 \label{eqn:S_asymptotically_expanded}
\lefteqn{S_{\chi\cdot e^{-i\lambda(\cdot)}}^{(\varpi)}\left(x_{1\lambda},x_{2\lambda}\right)^{(\ell)}}\\
&\sim&\frac{(2\pi)}{\pi^d}\,\lambda^{d-e/2}\,\varsigma_T(x)^{-(1+d)}\,d_\varpi\,
\overline{\chi_\varpi(g_\ell)}\,e^{i\sqrt{\lambda}\,\theta_1/\varsigma_T(x)
}\nonumber\\
&&\cdot\sum_{k\ge 0}\lambda^{-k/2}\int _{\mathbb{R}^e}\,\varrho_\lambda(\nu)\,
e^{i\,H^{(\ell)}_\nu\big(P_0^{(\ell)}(\nu)\big)}\,
Q_k\big(\theta_1,\mathbf{w}_1,\mathbf{w}_2,\nu\big)\,\mathrm{d}\nu\nonumber
\end{eqnarray}
for certain polynomials $Q_k$, of degree $\le 11\,k$ and parity $k$; we have $Q_0=1$.
Furthermore, $\varrho_\lambda(\nu)$ is again a bump
function
supported in an expanding ball of radius $O\left(\lambda^{1/24}\right)$.

Let us consider the Gaussian integrals in (\ref{eqn:S_asymptotically_expanded}).
First of all, for $k=0$ we have
\begin{eqnarray}
 \label{eqn:leading_gaussian_integral}
\lefteqn{\int _{\mathbb{R}^e}\,\varrho_\lambda(\nu)\,
e^{i\,H^{(\ell)}_\nu\big(P_0^{(\ell)}(\nu)\big)}\,
\mathrm{d}\nu\sim \int _{\mathbb{R}^e}\,
e^{i\,H^{(\ell)}_\nu\big(P_0^{(\ell)}(\nu)\big)}\,
\mathrm{d}\nu=e^{\mathcal{T}_\ell(\mathbf{w}_1,\mathbf{w}_2)} }\\
&&\cdot
\int_{\mathbb{R}^e}\exp\left(\frac{1}{\varsigma_T(x)}\,\left[
i\,\omega_0\big(\nu_M(m),\mathbf{w}_{1\mathrm{t}}+A_\ell\mathbf{w}_{2\mathrm{t}}\big)
-\frac 12\,\Big\|\nu_M(m)\Big\|^2\right]\right)\,\mathrm{d}\nu\nonumber.
\end{eqnarray}

Recall from the discussion preceding (\ref{eqn:identification_moment})
that $\nu\in \mathbb{R}^e$ represents the coordinates of $\xi\in \mathfrak{g}$ with respect to a chosen
orthonormal basis $(\xi_j)$ of $\mathfrak{g}$. If $C_m$ is as in \S \ref{sctn:effective_volumes}, we 
have
\begin{eqnarray}
 \label{eqn:exponent_rewritten}
\lefteqn{\frac{1}{\varsigma_T(x)}\,\left[i\,\omega_0\big(\nu_M(m),\mathbf{w}_{1\mathrm{t}}+A_\ell\mathbf{w}_{2\mathrm{t}}\big)
-\frac 12\,\Big\|\nu_M(m)\Big\|^2\right]}\\
&=&\frac{1}{\varsigma_T(x)}\,\left[-i\,\nu^t\,C^t\,J_0\,\big(\mathbf{w}_{1\mathrm{t}}+A_\ell\mathbf{w}_{2\mathrm{t}}\big)
-\frac 12\,\left\langle C\,\nu,C\,\nu\right\rangle\right]\nonumber\\
&=&-i\,\mathbf{a}^t\,J_0\,\frac{\big(\mathbf{w}_{1\mathrm{t}}+A_\ell\mathbf{w}_{2\mathrm{t}}\big)}{\sqrt{\varsigma_T(x)}}
-\frac 12\,\left\langle \mathbf{a},\mathbf{a}\right\rangle
\end{eqnarray}
where $J_0$ is the matrix of the standard complex structure, and in the latter line we have set
$\mathbf{a}=:C\,\nu/\sqrt{\varsigma_T(X)}$.

We then obtain, in view of (\ref{eqn:orbital_volume}),
\begin{eqnarray}
 \label{eqn:gaussian_integral}
\lefteqn{\int_{\mathbb{R}^e}\exp\left(\frac{1}{\varsigma_T(x)}\,\left[
i\,\omega_0\big(\nu_M(m),\mathbf{w}_{1\mathrm{t}}+A_\ell\mathbf{w}_{2\mathrm{t}}\big)
-\frac 12\,\Big\|\nu_M(m)\Big\|^2\right]\right)\,\mathrm{d}\nu}\\
&=&\varsigma_T(x)^{e/2}\,\dfrac{1}{\det(C_m)}\,\int_{\mathbb{R}^e}
\exp\left(-i\,\mathbf{a}^t\,J_0\,\frac{\big(\mathbf{w}_{1\mathrm{t}}+A_\ell\mathbf{w}_{2\mathrm{t}}\big)}{\sqrt{\varsigma_T(x)}}
-\frac 12\,\| \mathbf{a}\|^2\right)
\,\mathrm{d}\mathbf{a}\nonumber\\
&=&(2\pi)^{e/2}\,\varsigma_T(x)^{e/2}\,\dfrac{1}{|G^M_m|\,V_{\mathrm{eff}}^M(m)}\,\exp\left(-\frac{1}{2\,\varsigma_T(x)}\,
\|\mathbf{w}_{1\mathrm{t}}+A_\ell\mathbf{w}_{2\mathrm{t}}\|^2\right).\nonumber
\end{eqnarray}

Now
\begin{eqnarray}
 \label{eqn:total_exponent}
\lefteqn{\mathcal{T}_\ell(\mathbf{w}_1,\mathbf{w}_2)-\frac{1}{2\,\varsigma_T(x)}\,
\|\mathbf{w}_{1\mathrm{t}}+A_\ell\mathbf{w}_{2\mathrm{t}}\|^2}\\
&=&\frac{1}{\varsigma_T(x)}\,\left\{i \Big[\omega_0\big(\mathbf{w}_{1\mathrm{v}},\mathbf{w}_{1\mathrm{t}}\big)
-\omega_0\big(\mathbf{w}_{2\mathrm{v}},\mathbf{w}_{2\mathrm{t}}\big)\Big]+ 
\psi_2\big(\mathbf{w}_{1\mathrm{h}},A_\ell\mathbf{w}_{2\mathrm{h}}\big)\right.\nonumber\\
&&\left.-\big\|\mathbf{w}_{1\mathrm{t}}\big\|^2 -
\|A_\ell\mathbf{w}_{2\mathrm{t}}\|^2\right\}=
Q^T_{\mathrm{vt}}(\mathbf{w}_1,\mathbf{w}_2) 
+Q^T_\mathrm{h}(\mathbf{w}_1,A_\ell \mathbf{w}_2), \nonumber
\end{eqnarray}
since $\|A_\ell\mathbf{w}_{2\mathrm{t}}\|=\|\mathbf{w}_{2\mathrm{t}}\|$ by the unitarity of
$A_\ell$.

Thus the leading term in (\ref{eqn:S_asymptotically_expanded}) is
\begin{eqnarray}
 \label{eqn:leading_term_S}
\lefteqn{
\frac{(2\pi)^{1+e/2}}{\pi^d}\,\lambda^{d-e/2}\,\,\dfrac{1}{|G^M_m|\,V_{\mathrm{eff}}^M(m)}\,
d_\varpi\,\overline{\chi_\varpi(g_\ell)}
}\\
&&\cdot \varsigma_T(x)^{e/2-(1+d)}\,\exp\left(i\sqrt{\lambda}\,\frac{\theta_1}{\varsigma_T(x)}+Q^T_{\mathrm{vt}}(\mathbf{w}_1,\mathbf{w}_2) 
+Q^T_\mathrm{h}(\mathbf{w}_1,A_\ell \mathbf{w}_2)\right).\nonumber
\end{eqnarray}

More generally, 
let us write every $Q_k\big(\theta_1,\mathbf{w}_1,\mathbf{w}_2,\nu\big)$ in 
(\ref{eqn:S_asymptotically_expanded}) as a polynomial in $\nu$ with coefficients in
$\mathbb{C}[\theta_1,\mathbf{w}_1,\mathbf{w}_2]$.
If $\mathbf{a}^I$ is a mononomial of degree $|I|\ge 0$, then
\begin{eqnarray*}
 \int_{\mathbb{R}^e}\,\mathbf{a}^I\,
\exp\left(-i\,\mathbf{a}^t\,\mathbf{x}
-\frac 12\,\| \mathbf{a}\|^2\right)
\,\mathrm{d}\mathbf{a}=
P_I(\mathbf{x})\,e^{-\frac 12\,\|\mathbf{x}\|^2},
\end{eqnarray*}
where $P_I$ is a polynomial of degree $|I|$, in general non-homogeneous but of the same parity
as $|I|$. Therefore, the $k$-th term
in (\ref{eqn:S_asymptotically_expanded}) has the form
\begin{eqnarray}
 \label{eqn:k-th_term_S}
\lefteqn{
2\pi\cdot 2^{e/2}\,
\dfrac{1}{|G^M_m|\,V_{\mathrm{eff}}^M(m)}\,d_\varpi\,\overline{\chi_\varpi(g_\ell)}
\,\varsigma_T(x)^{-(1+d-e/2)}\,\left(\frac{\lambda}{\pi}\right)^{d-e/2}\,\lambda^{-k/2}}\\
&&\cdot R_k(\theta_1,\mathbf{w}_1,\mathbf{w}_2)\,\exp\left(i\sqrt{\lambda}\,\frac{\theta_1}{\varsigma_T(x)}+Q^T_{\mathrm{vt}}(\mathbf{w}_1,\mathbf{w}_2) 
+Q^T_\mathrm{h}(\mathbf{w}_1,A_\ell \mathbf{w}_2)\right),\nonumber
\end{eqnarray}
where $R_k$ is a polynomial of degree $k$ and parity $k$.
Theorem \ref{thm:local_scaling_asymtpotics} now follows by Remark \ref{rem:horizontality_effective}.
\end{proof}

\section{Proof of Corollary \ref{cor:trace_asymptotics}}

\begin{proof}
 Locally near any given $x\in X$, a HLC system centered at $x$ 
 may be deformed smoothly with $x$. More precisely,
 there exists an open neighborhood $Y\subseteq X$ of $x$ and a $\mathcal{C}^\infty$
 map 
 $$
 \Gamma:(y,\theta,\mathbf{v})\in Y\times (-\pi,\pi)\times B_{2d}(\mathbf{0},\delta)\mapsto
 \gamma_y(\theta,\mathbf{v})=y+(\theta,\mathbf{v}), 
 $$
 where $\gamma_y$ is a HLC system on $X$ centered at $y$. It may be assumed that $Y=\pi^{-1}(N)$, where
 $N\subseteq M$ is an open neighborhood of $m=:\pi(x)$, and that for any $y\in X_1$ we have
 $\big(y+(\theta,0)\big)+(\vartheta,\mathbf{v})=y+(\theta+\vartheta,\mathbf{v})$, when both sides are defined.
 
We may then find a finite open cover $\{M'_j\}$ of $M'$ such that, setting $X'_j=:\pi^{-1}\big(M'_j\big)$,
we have maps $\Gamma_j:X'_j\times (-\pi,\pi)\times B_{2d}(\mathbf{0},\delta)\rightarrow X$ as above.
For any $x\in X'_j$, recall the unitary isomorphism of $\mathbf{R}^e$ with the summand $\mathbb{R}^{e}_{\mathrm{ver}}$
in (\ref{eqn:explicit_direct_sum_decomposition}), which associated to any $\mathbf{w}\in \mathbf{R}^e$
a transverse vector $\mathbf{w}_\mathrm{t}\in \mathbb{R}^{e}_{\mathrm{ver}}\cong T_mM_{\mathrm{trasv}}$.

We obtain $S^1$-equivariant maps $\varLambda_j:X'_j\times B_{e}(\mathbf{0},\delta)\rightarrow X$, given by
\begin{equation}
 \label{eqn:local_diffeo_transverse}
 \varLambda_j(x,\mathbf{w})=:x+\mathbf{w}_\mathrm{t},
\end{equation}
and it is easily seen that these maps are local diffeomorphisms, and actually diffeomorphisms onto their images
$X_j=:\varLambda_j\big(X'_j\times B_{e}(\mathbf{0},\delta)\big)$
if $\delta>0$ is small enough. By the basic properties of HLC, they are furthermore
(point-wise) local isometries along $X'_j$, meaning that the pull back of the volume form
is
\begin{equation}
 \label{eqn:pull_back_volume_form}
 \varLambda_j^*(\mathrm{d}V_{X})=\mathcal{E}_j(x,\mathbf{w})\,\mathrm{d}\mathbf{w}\,\mathrm{d}V_{X'},
\end{equation}
where $\mathrm{d}V_{X'}$ is the volume form on $X'$ for the induced Riemannian structure (and orientation),
and $\mathcal{E}_j(x,\mathbf{0})=1$ identically.

Thus $\widetilde{X}=:\bigcup_j X_j$ is an open neighborhood of $X'$, 
$(X_j)$ is an  open cover of $\widetilde{X}$, and each $X_j$ is $S^1$-invariant. 
Furthermore, $(X'_j)$ is an open cover of $X'$, where $X'_j=:X_j\cap X'$.

Let $(\varrho_j)$ be a partition of unity on $\widetilde{X}$
subordinate to $(X_j)$; we may assume without loss that $\varrho_j$ is $S^1$-invariant,
hence the pull-back of a $\mathcal{C}^\infty$-function on $M$, that we shall still denote by $\varrho_j$. 
We may assume without loss that
$\varrho_j(x+\mathbf{w}_\mathrm{t})=\varrho_j(x-\mathbf{w}_\mathrm{t})$.

Also, let $\varrho\in \mathcal{C}^\infty_0\big(\widetilde{X}\big)$
be identically equal to $1$ on a small tubular neighborhood of $X'$. 

Then we have
\begin{eqnarray}
 \label{eqn:trace_computation}
 \lefteqn{\mathrm{trace}\left(S_{\chi\cdot e^{-i\lambda(\cdot)}}^{(\varpi)}\right)=
 \int_X\,S_{\chi\cdot e^{-i\lambda(\cdot)}}^{(\varpi)}(y,y)\,\mathrm{d}V_X(Y)}\\
 &\sim&\int_{\widetilde{X}}\varrho(y)\,S_{\chi\cdot e^{-i\lambda(\cdot)}}^{(\varpi)}(y,y)\,\mathrm{d}V_X(y)=
 \sum_j\int_{X_j}\varrho(y)\,\varrho_j(y)\,S_{\chi\cdot e^{-i\lambda(\cdot)}}^{(\varpi)}(y,y)\,\mathrm{d}V_X(y)
 \nonumber\\
 &=&\sum_j\int_{X_j'}\int_{B_{e}(\mathbf{0},\delta)}\varrho(x+\mathbf{w}_t)\,
 \varrho_j(x+\mathbf{w}_t)\,S_{\chi\cdot e^{-i\lambda(\cdot)}}^{(\varpi)}(x+\mathbf{w}_t,x+\mathbf{w}_t)\nonumber\\
&&\cdot \mathcal{E}_j(x,\mathbf{w})\,\mathrm{d}\mathbf{w}\,\mathrm{d}V_{X'}(x)\nonumber\\
&=&\lambda^{-e/2}\,\sum_j\int_{X'_j}\int_{B_{e}(\mathbf{0},\delta)}\,
\varrho\left(x+\frac{\mathbf{w}_t}{\sqrt{\lambda}}\right)\,
 \varrho_j\left(x+\frac{\mathbf{w}_t}{\sqrt{\lambda}}\right)\nonumber\\
&&\cdot 
S_{\chi\cdot e^{-i\lambda(\cdot)}}^{(\varpi)}\left(x+\frac{\mathbf{w}_t}{\sqrt{\lambda}},x+\frac{\mathbf{w}_t}{\sqrt{\lambda}}\right)\,
\mathcal{E}_j\left(x,\dfrac{\mathbf{w}}{\sqrt{\lambda}}\right)\,\mathrm{d}\mathbf{w}\,\mathrm{d}V_{X'}(x),\nonumber
\end{eqnarray}
where we have performed the change of coordinates $\mathbf{w}\mapsto \mathbf{w}/\sqrt{\lambda}$.

We can now make use of the local asymptotic expansion from Corollary \ref{cor:local_weyl_law}.
Using that
\begin{equation*}
 \int_{\mathbb{R}^e}\exp\left(-\frac{2}{\varsigma_T(x)}\,\|\mathbf{w}\|^2\right)\,\mathrm{d}\mathbf{w}=
\left(\frac{\pi}{2}\right)^{e/2}\,\varsigma_T(x)^{e/2},
\end{equation*}
and in view of the parity statement on the $F_j$'s, we get
\begin{eqnarray*}
\mathrm{trace}\left(S_{\chi\cdot e^{-i\lambda(\cdot)}}^{(\varpi)}\right)
\sim
2\pi\cdot \dim(V_\varpi)\,\left(\frac{\lambda}{\pi}\right)^{d-e}\cdot \int_{X'}F(\lambda,x)\,\mathrm{d}V_{X'}(x),
\end{eqnarray*}
where $F(\lambda,\cdot):X'\rightarrow \mathbb{R}$ has an asymptotic expansion
$$
F(\lambda,x)\sim
\frac{a_{\Phi,\varpi}(m)}{V^X_{\mathrm{eff}}(m)}\cdot\varsigma_T(x)^{-(d+1-e)}\cdot\left[1+\sum_{j\ge 1}\lambda^{-j}\,\beta_j(x)\right];
$$
here $a_{\Phi,\varpi}$ is as in Definition \ref{defn:a_varpi_Phi}.
We conclude
\begin{eqnarray*}
\mathrm{trace}\left(S_{\chi\cdot e^{-i\lambda(\cdot)}}^{(\varpi)}\right)
&\sim&
2\pi\cdot \dim(V_\varpi)\,\left(\frac{\lambda}{\pi}\right)^{d-e}\,a_{\mathrm{gen}}(\Phi,\varpi)\,\Gamma (\Phi,\varpi)\\
&&\cdot \left[1+\sum_{j\ge 0}\lambda^{-j}\,E_j\right],
\end{eqnarray*}
where
$$
\Gamma (\Phi,\varpi)=:\int_{X'}\frac{1}{V^X_{\mathrm{eff}}(m)}\,\varsigma_T(x)^{-(d+1-e)}\,\mathrm{d}V_{X'}(x).
$$
\end{proof}

\section{Proof of Corollary \ref{cor:weyl_law}}

\begin{proof}
As in the proof of Theorem 1.1 of \cite{p_weyl}, we shall adapt the classical Tauberian argument in \S 12 of \cite{gr_sj}. 
Given Corollary \ref{cor:trace_asymptotics}, we have that  for $\lambda\gg 0$
$$N_T^{(\varpi)}(\lambda+1)-N_T^{(\varpi)}(\lambda)\le C\,\lambda^{d-e},$$ and therefore for any $\tau$ and
$\lambda>1$
$$
\big | N_T^{(\varpi)}(\lambda-\tau)-N_T^{(\varpi)}(\lambda)\big|\le C'\,\big(1+|\tau|\big)^{d-e+1}\,\lambda^{d-e}.
$$
Hence
\begin{equation}
 \label{eqn:key_weyl_estimate}
\int_{-\infty}^{+\infty}\left[N_T^{(\varpi)}(\lambda-\tau)-N_T^{(\varpi)}(\lambda)\right]\,\chi(\tau)\,\mathrm{d}\tau
=O\left(\lambda^{d-e}\right).
\end{equation}

Let us consider the spectral measure $\mathrm{d}\mathcal{T}^{(\varpi)}=\sum_j\delta_{\lambda_j^{(\varpi)}}$ on $\mathbb{R}$, 
where $\delta_a$ is Dirac's delta at $a\in \mathbb{R}$.
Thus
$$
N_T^{(\varpi)}(\lambda)=\int_{-\infty}^\lambda \,\mathrm{d}\mathcal{T}^{(\varpi)}(\eta),\,\,\,\,\,\,\,\,\,\,\,\,\,\,
\sum_j\widehat{\chi}\big(\lambda-\lambda_j^{(\varpi)}\big)
=\int_{-\infty}^{+\infty}\,\widehat{\chi}(\lambda-\eta)\,\mathrm{d}\mathcal{T}^{(\varpi)}(\eta).
$$

Now we set $G(\lambda)=:\int_{-\infty}^\lambda \widehat{\chi}(\tau)\,\mathrm{d}\tau$, and compute 
$\int_{-\infty}^{+\infty}\,G(\lambda-\eta)\,\mathrm{d}\mathcal{T}^{(\varpi)}(\eta)$ in two different manners.
 On the one hand, by change of variable and the Tonelli-Fubini Theorem,
\begin{eqnarray}
 \label{eqn:1st_manner}
\lefteqn{\int_{-\infty}^{+\infty}\,G(\lambda-\eta)\,\mathrm{d}\mathcal{T}^{(\varpi)}(\eta)}\\
&=& \int_{-\infty}^{+\infty}\,\left[\int_{-\infty}^{\lambda-\eta}\widehat{\chi}(\tau)\,\mathrm{d}\tau\right]\,\mathrm{d}\mathcal{T}^{(\varpi)}(\eta)
= \int_{-\infty}^{+\infty}\,\left[\int_{-\infty}^{\lambda}\widehat{\chi}(\tau-\eta)\,\mathrm{d}\tau\right]\,\mathrm{d}\mathcal{T}^{(\varpi)}(\eta)\nonumber\\
&=&\int_{-\infty}^{\lambda} \,\left[\int_{-\infty}^{+\infty}\,\widehat{\chi}(\tau-\eta)\,\mathrm{d}\mathcal{T}^{(\varpi)}(\eta)\right]
\, \mathrm{d}\tau=\int_{-\infty}^{\lambda} \,\sum_j \widehat{\chi}\big(\tau-\lambda_j^{(\varpi)}\big)\,
\mathrm{d}\tau  \nonumber\\
&=&\int_{-\infty}^{\lambda} \,\mathrm{trace}\left(S_{\chi\cdot e^{-i\tau(\cdot)}}^{(\varpi)}\right)\,
\mathrm{d}\tau\nonumber\\
&=& \frac{2\pi^2}{d-e+1}\cdot \dim(V_\varpi)\,a_{\mathrm{gen}}(\Phi,\varpi)\,\Gamma (\Phi,\varpi)\, \left(\frac{\lambda}{\pi}\right)^{d-e+1}
+O\left(\lambda^{d-e}\right).\nonumber
\end{eqnarray}

On the other hand, letting $H$ denote the Heavyside function, we also have:
\begin{eqnarray}
\label{eqn:2nd_manner}
 \lefteqn{\int_{-\infty}^{+\infty}\,G(\lambda-\eta)\,\mathrm{d}\mathcal{T}^{(\varpi)}(\eta)
 =\int_{-\infty}^{+\infty}\,\left[\int_{-\infty}^{\lambda-\eta} 
 \widehat{\chi}(\tau)\,\mathrm{d}\tau\right]\,\mathrm{d}\mathcal{T}^{(\varpi)}(\eta)}\nonumber\\
 &=&
 \int_{-\infty}^{+\infty}\,\left[\int_{-\infty}^{+\infty}H(\lambda-\eta-\tau)\, 
 \widehat{\chi}(\tau)\,\mathrm{d}\tau\right]\,\mathrm{d}\mathcal{T}^{(\varpi)}(\eta)\nonumber\\
  &=&
 \int_{-\infty}^{+\infty}\,\left[\int_{-\infty}^{+\infty}H(\lambda-\eta-\tau)\, 
 \mathrm{d}\mathcal{T}^{(\varpi)}(\eta)\right]\,\widehat{\chi}(\tau)\,\mathrm{d}\tau\nonumber\\
 &=&\int_{-\infty}^{+\infty}\,N_T^{(\varpi)}(\lambda-\tau)\,\widehat{\chi}(\tau)\,\mathrm{d}\tau\nonumber\\
  &=&N_T^{(\varpi)}(\lambda)\,\int_{-\infty}^{+\infty}\,\widehat{\chi}(\tau)\,\mathrm{d}\tau+
 \int_{-\infty}^{+\infty}\,\left[N_T^{(\varpi)}(\lambda-\tau)-N_T^{(\varpi)}(\lambda)\right]
 \,\widehat{\chi}(\tau)\,\mathrm{d}\tau \nonumber\\
  &=&2\pi\,N_T^{(\varpi)}(\lambda)+
 O\left(\lambda^{d-e}\right),
\end{eqnarray}
where in the last line we have used (\ref{eqn:key_weyl_estimate}) and that
$$
1=\chi(0)=\frac{1}{2\pi}\,\int_{-\infty}^{+\infty}\widehat{\chi}(\tau)\,\mathrm{d}\tau
$$
by our choice of $\chi$ and the Fourier inversion formula.
Comparing (\ref{eqn:1st_manner}) and (\ref{eqn:2nd_manner}), we conclude that
$$
N_T^{(\varpi)}(\lambda)=
\frac{\pi}{d-e+1}\cdot \dim(V_\varpi)\,a_{\mathrm{gen}}(\Phi,\varpi)\,\Gamma (\Phi,\varpi)\, 
\left(\frac{\lambda}{\pi}\right)^{d-e+1}
+O\left(\lambda^{d-e}\right),
$$
as claimed.

\end{proof}

\section{Proof of Corollary \ref{cor:weyl_law_volume}}

\begin{proof}
By first integrating along the orbits of $\widetilde{\mu}$
 on $X'$
 (that is, the fibers of the projection $X'\rightarrow \widehat{X}$) and then on the base $\widehat{X}$, we
 have in view of (\ref{eqn:Gamma_Phi_T}):
 \begin{equation}
  \label{eqn:volume_quotient}
  \Gamma (\Phi,\varpi)=\int_{\widehat{X}}\,
  \widehat{\varsigma}_T\big(\widehat{x}\big)^{-(d-e+1)}\,\mathrm{d}V_{\widehat{X}}\big(\widehat{x}\big).
 \end{equation}
Furthermore, we have
$$
\widehat{\Sigma}=:\left\{\left(\widehat{x},r\,\widehat{\alpha}_{\widehat{x}}\right)\,:\,
\widehat{x}\in \widehat{X},\,r>0\right\},
$$
and the symplectic form on $\widehat{\Sigma}$ is given by the analogue of (\ref{eqn:symplectic_Sigma}):
$$
\omega_{\widehat{\Sigma}}=\mathrm{d}\left(r\,\widehat{\alpha}\right)
=\mathrm{d}r\wedge \widehat{\alpha}+r\,\mathrm{d}\widehat{\alpha}=
\mathrm{d}r\wedge \widehat{\alpha}+2r\,\widehat{\omega},
$$
where we omit the pull-back symbols in front of $\widehat{\alpha}$ and $\widehat{\omega}$.

Since $\widehat{\Sigma}$ has dimension $2\,(d-e+1)$, the symplectic volume form is
\begin{eqnarray*}
\mathrm{d}V_{\widehat{\Sigma}}&=& \frac{1}{(d-e+1)!}\,\omega_{\widehat{\Sigma}}^{\wedge (d-e+1)}\\
&=&
(2\,r)^{d-e}\,\frac{1}{(d-e)!}\,\omega_{\widehat{M}}^{\wedge (d-e)}\wedge \mathrm{d}r\wedge \widehat{\alpha}
=(2\,r)^{d-e}\,\mathrm{d}V_{\widehat{M}}\wedge \mathrm{d}r\wedge \widehat{\alpha}\\
&=&-2^{d-e+1}\,\pi\,r^{d-e}\,\mathrm{d}V_{\widehat{X}}\wedge \mathrm{d}r.
\end{eqnarray*}

Since $\widehat{\sigma}_T=r\,\widehat{\varsigma}_T$,
$$
\widehat{\Sigma}_1=\left\{\left(\widehat{x},r\,\widehat{\alpha}_{\widehat{x}}\right)\in \widehat{\Sigma}\,:\,
\widehat{x}\in \widehat{X},\,r<1/\widehat{\varsigma_T}\big(\widehat{x}\big)\right\}
$$
and so in view of (\ref{eqn:volume_quotient}) its symplectic volume is
\begin{eqnarray*}
\mathrm{vol}\left(\widehat{\Sigma}_1\right)&=&2^{d-e+1}\pi\,\int_{\widehat{X}}
\left[\int_0^{1/\widehat{\varsigma_T}\big(\widehat{x}\big)}\,r^{d-e}\,\mathrm{d}r\right]\,\mathrm{d}V_{\widehat{X}}\\
&=&\frac{2^{d-e+1}\pi}{d-e+1}\,\int_{\widehat{X}}\,\widehat{\varsigma_T}\left(\widehat{x}\right)^{-(d-e+1)}
\,\mathrm{d}V_{\widehat{X}}=
\frac{2^{d-e+1}\pi}{d-e+1}\,\Gamma (\Phi,\varpi).
\end{eqnarray*}
Inserting this in the estimate of Corollary \ref{cor:weyl_law} we obtain
$$
N_T^{(\varpi)}(\lambda)=
\dim(V_\varpi)\,a_{\mathrm{gen}}(\Phi,\varpi)\,\mathrm{vol}\left(\widehat{\Sigma}_1\right)\, 
\left(\frac{\lambda}{2\,\pi}\right)^{d-e+1}
+O\left(\lambda^{d-e}\right),
$$
as claimed, since $a_{\mathrm{gen}}(\Phi,\varpi)=\dim(V_\varpi)$ in this case.
\end{proof}

\end{document}